\documentclass[a4paper,12pt]{article}

\usepackage{amsmath,amsthm,amssymb}
\usepackage{algorithmic}
\usepackage{graphicx}
\usepackage{enumitem}
\usepackage{xspace}
\usepackage{url}

\DeclareMathOperator{\lcm}{lcm}
\DeclareMathOperator{\lm}{lm}
\DeclareMathOperator{\lc}{lc}
\DeclareMathOperator{\NF}{NF}

\DeclareMathOperator{\spoly}{spoly}

\DeclareMathOperator{\Ker}{ker}
\DeclareMathOperator{\syz}{syz}
\DeclareMathOperator{\Ann}{Ann}
\DeclareMathOperator{\ann}{Ann}

\DeclareMathOperator{\ini}{in}
\DeclareMathOperator{\codim}{codim}

\newcommand{\lra}{\longrightarrow}

\newcommand{\ra}{\rightarrow}

\newcommand{\N}{{\mathbb N}}
\newcommand{\R}{{\mathbb R}}
\newcommand{\K}{{\mathbb K}}
\newcommand{\Z}{{\mathbb Z}}

\newcommand{\C}{{\mathbb C}}

\renewcommand{\d}{\partial }
\newcommand{\BS}{Bernstein-Sato\xspace}
\newcommand{\BM}{Brian\c{c}on-Maisonobe\xspace}
\newcommand{\eg}{e.~g.\xspace}
\newcommand{\ie}{i.~e.\xspace}
\newcommand{\comment}[1]{}
\newcommand{\doublequote}[1]{``#1''}      
\newcommand{\inw}{\ini_{(-w,w)}}
\newcommand{\inwh}{\ini_{(-w,w,0)}}
\renewcommand{\algorithmicrequire}{\textbf{Input:}} 
\renewcommand{\algorithmicensure}{\textbf{Output:}} 
\newcommand{\vars}{\sigma}                

\theoremstyle{plain}
\newtheorem{lemma}{Lemma}[section]
\newtheorem{proposition}[lemma]{Proposition}
\newtheorem{theorem}[lemma]{Theorem}
\newtheorem{corollary}[lemma]{Corollary}

\theoremstyle{definition}
\newtheorem{definition}[lemma]{Definition}
\newtheorem{example}[lemma]{Example}
\newtheorem{algorithm}[lemma]{Algorithm}

\newtheorem{remark}[lemma]{Remark}

\author{Daniel Andres \and Michael Brickenstein \and Viktor Levandovskyy \and Jorge Mart\'{i}n-Morales \and Hans Sch\"onemann}

\title{Constructive $D$-module Theory with {\sc Singular}}

\date{}

\begin{document}

\maketitle

\begin{abstract}
We overview numerous algorithms in computational $D$-module theory together with the theoretical background as well as the implementation in the computer algebra system \textsc{Singular}.~We discuss new approaches to the computation of Bernstein operators, of logarithmic annihilator of a polynomial, of annihilators of rational functions as well as complex powers of polynomials. We analyze algorithms for local Bernstein-Sato polynomials and also algorithms, recovering any kind of Bernstein-Sato polynomial from partial knowledge of its roots. We address a novel way to compute the \BS polynomial for an affine variety algorithmically. All the carefully selected nontrivial examples, which we present, have been computed with our implementation. We address such applications as the computation of a zeta-function for certain integrals and revealing the algebraic dependence between pairwise commuting elements.

\vspace{0.15cm}

\vspace{0.1cm}
\end{abstract}

\section{Introduction}

Constructive $D$-module theory has been dynamically developing throughout the last years. There are new approaches, algorithms, implementations and applications. Our work on the implementation of procedures for $D$-modules started in 2003, motivated among other factors by challenging elimination problems in non-com\-mu\-ta\-tive algebras, which appear \eg in algorithms for computing Bernstein-Sato polynomials. We reported on solving several challenges in \cite{LM08}. A non-commutative subsystem \textsc{Singular:Plural} \cite{Plural} of the computer algebra system \textsc{Singular} provides a user with possibilities to compute numerous Gr\"obner bases-based procedures in a wide class  of non-com\-mu\-ta\-tive $G$-algebras \cite{LS03}. It was natural to use this functionality in the context of computational $D$-module theory. Nowadays we present a $D$-module suite in \textsc{Singular} consisting of the libraries \texttt{dmod.lib}, \texttt{dmodapp.lib}, \texttt{dmodvar.lib} and \texttt{bfun.lib}. There are many useful and flexible procedures for various aspects of $D$-module theory. These libraries are freely distributed together with \textsc{Singular} \cite{Singular}.

There are several implementations of algorithms for $D$-modules, namely the experimental program \textsc{kan/sm1} by N.~Takayama \cite{KAN}, the \texttt{bfct} package in \textsc{Risa/Asir} \cite{Asir} by M.~Noro \cite{Noro02} and the package \texttt{Dmodules.m2} in \textsc{Macaulay2} by A.~Leykin and H.~Tsai \cite{dmodMac2}. We aim at creating a $D$-module suite, which will combine flexibility and rich functionality with high performance, being able to treat more complicated examples.

In this paper we do not present any comparison between different computer algebra systems in the realm of $D$-modules, referring to \cite{LM08} and \cite{ALM09}. However, comparison in the latter articles shows, that our implementation is superior to \textsc{kan/sm1} and \textsc{Macaulay2} and in many cases more powerful than \textsc{Risa/Asir}.

Here is the list of problems we address in this paper:

\begin{itemize}
\item $s$-parametric annihilator of $f$ (Section \ref{sAnn}, see also \cite{LM08,ALM09}),
\item annihilator of $f^{\alpha}$ for $\alpha \in \C$ (Section \ref{anniA}, see also \cite{SST00}),
\item annihilator of a polynomial function $f$ and of a rational function $f/g$ (Section \ref{anniA}),
\item $b$-function with respect to weights for an ideal (Section \ref{bfctIdeal}, see also \cite{ALM09}),
\item global and local Bernstein-Sato polynomials of $f$ (Section \ref{bspolyA}),
\item partial knowledge of \BS polynomial (Section \ref{checkrootA}, see also \cite{LM08}),
\item Bernstein operator of $f$ (Section \ref{Boperator}),
\item logarithmic annihilator of $f$ (Section \ref{logannA}),
\item Bernstein-Sato ideals for $f = f_1 \cdot \ldots \cdot f_m$ (Section \ref{BMI}, see also \cite{LM08}),
\item annihilator and Bernstein-Sato polynomial for a variety (Section \ref{bsvarA}, see also \cite{ALM09}).
\end{itemize}

We describe both theoretical and implementational aspects of the problems above and illustrate them with carefully selected nontrivial examples, computed with our implementation in \textsc{Singular}. In Section \ref{newBM}, we give yet another alternative proof for the algorithm by \BM  for computing $\ann_{D_n[s]}(f^s)$, presented in \cite{ALM09}. Notably, this delivers additional structural information. In Section \ref{Boperator}, we compare several approaches for the computation of Bernstein operators. Using the method of principal intersection, we formalize several methods for computing \BS polynomials. Following Budur et al. \cite{BMS06} and \cite{ALM09}, we report on the implementation of two methods for the computation of \BS polynomials for affine varieties in a framework, which is a natural generalization of our approach to the algorithm by \BM.

\textit{Notations}. Throughout the article $\K$ is assumed to be a field of characteristic zero. By $R$ we denote the polynomial ring $\K[x_1,\ldots,x_n]$ and by $f\in R$ a non-constant polynomial.

We consider the $n$-th Weyl algebra as the algebra of linear partial differential operators with polynomials coefficients. That is $D_n = D(R) = \K\langle x_1,\dots,x_n,\partial_1,\dots,\partial_n \mid \{ \partial_i x_i = x_i \partial_i +1, \partial_i x_j = x_j \partial_i, i \not=j \}  \rangle$. We denote by $D_n[s] = D(R) \otimes_{\K} \K[s_1,\ldots,s_n]$ and drop the index $n$ depending on the context. 

The ring $R$ is a natural $D_n(R)$-module with the action
\[
x_i \bullet f(x_1,\ldots,x_n) = x_i \cdot f(x_1,\ldots,x_n), \quad 
\d_i \bullet f(x_1,\ldots,x_n) = \frac{\d f(x_1,\ldots,x_n)}{\d x_i}. 
\]

Working with monomial orderings in elimination, we use the notation $x \gg y$ for \doublequote{$x$ is greater than any power of $y$}.

Given an associative $\K$-algebra $A$ and some monomial well-ordering on $A$, we denote by $\lm(f)$ (resp. $\lc(f)$) the leading monomial (resp. the leading coefficient) of $f\in A$. Given a left Gr\"obner basis $G \subset A$ and $f\in A$, we denote by $\NF(f,G)$ the normal form of $f$ with respect to the left ideal ${}_{A}\langle G \rangle$. We also use the shorthand notation $h \ra_{H} f$ (and $h \ra f$, if $H$ is clear from the context) for the reduction of $h\in A$ to $f\in A$ with respect to the set $H$. If not specified, under \textit{ideal} we mean \textit{left ideal}. For $a,b \in A $, we use the Lie bracket notation $[a,b]:=ab - ba$ as well as the skew Lie bracket notation $[a,b]_{k}:=ab - k\cdot ba$ for $k\in\K^{*}$.

It is convenient to treat the algebras we deal with in a bigger framework of $G$-algebras of Lie type.

\begin{definition}\label{GalgLie}
Let  $A$ be the quotient of the free associative algebra $\K\langle x_1,\ldots,x_n\rangle$ by the two-sided ideal $I$, generated by the finite set $\{x_jx_i - x_ix_j-d_{ij} \} \; \forall 1\leq i<j \leq n$, where $d_{ij} \in \K[x_1,\ldots,x_n]$. $A$ is called a {\em $G$--algebra of Lie type} \cite{LS03}, if\\
$\bullet$ \; $\forall \; 1\leq i < j < k \leq n$ the expression $d_{ij}x_k - x_k d_{ij} + x_j d_{ik} - d_{ik}x_j + d_{jk}x_i - x_i d_{jk}$ reduces to zero modulo $I$ and,\\
$\bullet$ there exists a monomial ordering $\prec$ on $\K[ x_1,\ldots,x_n]$, such that $\lm(d_{ij}) \prec x_i x_j$, $\forall i < j$.
\end{definition}

$G$-algebras are also known as \textit{algebras of solvable type} \cite{KW,HLi} and \textit{PBW algebras} \cite{BGV}. We often use the following.

\begin{lemma}[Generalized Product Criterion, \cite{LS03}]\label{prodCrit}
Let $A$ be a $G$-algebra of Lie type and $f,g\in A$. Suppose $\lm(f)$ and $\lm(g)$ have no common factors, then $\spoly(f,g) \ra_{\{f,g\}} [f,g]$.
\end{lemma}

\section{Challenges for Gr\"obner bases engines of \textsc{Singular} }

Since the very beginning of implementation of algorithms for $D$-modules in \textsc{Singular} there have been intensive interaction with the developers of \textsc{Singular}. Numerous challenging examples and open problems from constructive $D$-module theory were approached both on the level of libraries and in the kernel of \textsc{Singular} and \textsc{Singular:Plural}. This resulted in several enhancements in kernel procedures and, among other, motivated M.~Brickenstein to develop and implement the generalization of his \texttt{slimgb} \cite{SlimBa06} (slim Gr\"obner basis) algorithm to non-commutative $G$-algebras. Indeed \texttt{slimgb} is a variant of Buchberger's algorithm. It is designed to keep polynomials \textit{slim}, that is short with small coefficients. The algorithm features parallel reductions and a strategy to minimize the weighted lengths of polynomials. A weighted length function of a polynomial can be seen as measure for the intermediate expression swell  and it can consider not only the number of terms in a polynomial, but also their coefficients and degrees. Considering the degrees of the terms inside the polynomials, \texttt{slimgb} can often directly (that is, without using Gr\"obner Walk or similar algorithms) compute Gr\"obner bases with respect to \eg elimination orderings. The procedure \texttt{slimgb} demonstrated very good performance on examples from the realm of $D$-modules \cite{LM08}, which require computations with elimination orderings. 

As it will be seen in the paper, various computational questions, arising in $D$-module theory, use much more than Gr\"obner bases only. Among other, a transformation matrix between two bases (called \textsc{Lift} in \cite{GPS08}), the kernel of a module homomorphism (called \textsc{Modulo} in \cite{GPS08}) and so on must be applied for complicated examples. On the other hand, the standard \texttt{std} routine for Gr\"obner bases, generalized to non-commutative $G$-algebras, is used together with \texttt{slimgb} for a variety of problems. Since the beginning of development of the $D$-module suite in \textsc{Singular}, these functions have been enhanced: they became much faster and more flexible. The effect of the use of the generalized Chain Criterion (cf. \cite{GPS08}) in Gr\"obner engines is even bigger in the non-commutative case, due to the discarding of multiplications, which complexity is increased, compared with the commutative case. On the contrary, the generalized Product Criterion (Lemma \ref{prodCrit}) plays a minor role in the implementation, since the complete discarding of a pair generalizes to the computation of a Lie bracket of the pair members. 

The concept of \textit{ring list}, introduced in \textsc{Singular} in 2004, enormously simplified the process of creation and modification of rings (like changing the monomial module ordering, regrouping of variables, modifying non-commutative relations, working with parameters of the ground field etc.). Especially in the $D$-module setting we modify rings often, create a new one from existing rings and equip a new ring with a new ordering. Thus, with ring lists the development of such procedures became much easier and the corresponding code became much more manageable.

We have to mention, that in the meantime the implementation of non-commutative multiplication in the kernel of \textsc{Singular:Plural} has been improved as well.

\section{$s$-parametric annihilator of $f$} \label{sAnn}

Recall Malgrange's construction for $f = f_1 \cdots f_p\in \K[x_1,\ldots,x_n]$. Consider the algebra $W_{p+n}$, being the $(p+n)$-th Weyl algebra
\[
\K\langle  \{t_j, {\partial t}_j \mid 1\leq j \leq p\} \mid\{  [{\partial t}_j, t_k]=\delta_{jk} \} \rangle \otimes_{\K} \K\langle \{x_i, \d_i \mid 1\leq i \leq n\} \mid\{ [\d_i, x_k]=\delta_{ik} \} \rangle.
\]
Moreover, consider the left ideal in $W_{p+n}$, called Malgrange ideal 
\[
I_f  := \langle \; \{ \; t_j - f_j, \sum^p_{j=1} \frac{\partial f_j}{\partial x_i} {\partial t}_j + \partial_i \;\mid 1\leq j\leq p, 1\leq i\leq n \}\; \rangle.
\]

Then for $s=(s_1,\ldots,s_p)$ we denote $f^s := f_1^{s_1} \cdots f_p^{s_p}$. Let us compute
\[
 I_f \cap \K[\{t_j \d t_j\}]\langle x_i, \d x_i \mid [\d_i, x_i]=1 \rangle \subset D_n[\{t_j \d t_j \}]
\]
and furthermore, replace $t_j \d t_j$ with $-s_j-1$. The result is known (\eg \cite{SST00}) to coincide with the $s$-parametric annihilator of $f^s$, $\Ann_{D_n[s]}(f^s) = \{ Q(x,\d, s) \in D_n[s] \mid Q\bullet f^s = 0\}$.

There exist several methods for the computation of $s$-parametric annihilator of $f^s$.

\subsection{Oaku and Takayama}

The algorithm by Oaku and Takayama \cite{Oaku97,SST00} was developed in a wider context and uses homogenization. Consider the $\K$-algebras $T:=\K[t_1,\ldots,t_p]$, $D'_p:= D(T)$ and $H := D_n \otimes _\K D'_p \otimes _\K \K[u_1,\ldots,u_p,v_1,\ldots,v_p]$. Moreover, let $I$ below be the $(u,v)$-homogenized Malgrange ideal, that is the left ideal in $H$
\[
I = \left\langle  \{t_j - u_j f_j, \sum^p_{k=1} \frac{\d f_k}{\d x_i} u_k {\d t}_j + \partial_i, u_j v_j -1 \} \right\rangle.
\]
Oaku and Takayama proved, that $\ann_{D_n[s]}(f^s)$ can be obtained in two steps. At first $\{u_j, v_j\}$ are eliminated from $I$ with the help of Gr\"obner bases, thus yielding $I' = I \cap (D_n \otimes _\K D'_p)$. Then, one calculates $I' \cap (D_n \otimes _\K \K[\{-t_j \d t_j -1 \}])$ and substitutes every appearance of $t_j \d t_j$ by $-s_j-1$ in the latter. The result is then $\ann_{D_n[s]}(f^s)$.

\subsection{Brian\c{c}on and Maisonobe}
\label{BM}

Consider $S_p = \K\langle \{ \d t_j, s_j \} \mid \d t_j s_k = s_k \d t_j - \delta_{jk} \d t_j  \rangle$ (the $p$-th shift algebra) and $S' = D_n \otimes_{\K} S_p$. Moreover, consider the following left ideal in $S'$:
\[
I = \left\langle \{s_j +  f_j \d t_j, \sum^p_{k=1} \frac{\partial f_k}{\partial x_i} {\partial t}_k + \partial_i \} \right\rangle.
\]
Brian\c{c}on and Maisonobe proved \cite{BM02} that $\ann_{D_n[s]}(f^s) = I\cap D_n[s_1,\ldots,s_p]$ and hence the latter can be computed via the left Gr\"obner basis with respect to an elimination ordering for $\{\d t_j\}$.

\subsection{Another alternative proof of Brian\c{c}on-Maisonobe's method}\label{newBM}

Here we give yet another \cite{ALM09} computer algebraic proof for the method by Brian\c{c}on and Maisonobe.

Throughout this section, we assume $1\leq i \leq n$ and $1\leq j \leq p$. Define 
\[
E := \K\langle \{t_j, \d t_j, x_i, \d_i, s_j \}\mid \{[\d_i, x_i]=1,  [\d t_j, t_j]=1, 
[t_k, s_j] = \delta_{jk} t_j, [\d t_k, s_j] = -\delta_{jk} \d t_j\} \rangle.
\]
Let $B = D_n[s]$ be a subalgebra of $E$, generated by $\{ x_i, \d_i, s_j \}$. Then the \BM method requires to prove \cite{ALM09}, that
\[
\langle \{t_j - f_j, \sum^p_{j=1} \frac{\partial f_j}{\partial x_i} {\partial t}_j + \partial_i,  f_j \partial t_j + s_j \} \rangle \cap D_n[s] = \ann_{D_n[s]}(f^s).
\]

\begin{theorem}
Let us define the following polynomials and sets:
\[
g_i :=\partial_i + \sum^p_{k=1} \frac{\partial f_k}{\partial x_i} {\partial t}_k, \
G = \{g_i\}, T = \{t_j - f_j\}, S = \{s_j +  f_j \d t_j \}.
\] 
Let $\Lambda$ be a (possibly empty) subset of $\{1,\ldots,p\}$. Define $M_\Lambda:=G \cup \{ t_k - f_k \mid k \in \Lambda\} \cup \{s_j +  f_j \d t_j \mid j \in \{1,\ldots,p\} \setminus\Lambda \}$.
\begin{enumerate} [label=(\alph*)]
\item \label{thm.a} For any $\Lambda$, the elements of $M_\Lambda$ commute pairwise. In particular, so do $G \cup T$ and~$G \cup S$.
\item \label{thm.b} Consider an ordering $\prec$, satisfying $\{t_j \}  \gg  \{ x_i \} $, $\{\d_i, s_j \}  \gg  \{ x_i, \d t_j \}$.  Then any subset of $G \cup T \cup S$ is a left Gr\"obner basis with respect to $\prec$. In particular, so is the set $M_\Lambda$ for any $\Lambda$.
\item \label{thm.c} The elements of $M_\Lambda$ are algebraically independent.
\item \label{thm.d}
For any $\Lambda$, the Krull (and hence the Gel'fand-Kirillov) dimension of $\K[M_\Lambda]$ is $n+p$. 
\item \label{thm.e} For any $\Lambda$, $\K[M_\Lambda]$ is a maximal commutative subalgebra of $E$.
\end{enumerate}
\end{theorem}

\begin{proof}
\ref{thm.a} Computing commutators between elements, we obtain 
\begin{gather*}
[g_i,g_k] = {\d t}_j \sum_j [\d_i, \frac{\d f_j}{\d x_k} ] +  {\d t}_j \sum_j [\frac{\d f_j}{\d x_i}, \d_k] 
 = {\d t}_j \sum_j ([\d_i, \frac{\d f_j}{\d x_k} ] - [\d_k, \frac{\d f_j}{\d x_i} ]) = 0,\\
[t_k - f_k,g_i] = \sum_j \frac{\d f_j}{\d x_i} [t_k,\d t_j] - [f_k, \d_i]=0, \qquad
[t_i - f_i,t_k - f_k] = 0,\\
[s_i +  f_i \d t_i , s_j +  f_j \d t_j] = f_j [s_i, \d t_j] - f_i [s_j, \d t_i] = 0, \qquad
[s_j + f_j \d t_j,g_i] = [s_j, \d_i] +  \\ + \d t_j [f_j, \d_i] + \sum^p_{k=1} \frac{\d f_k}{\d x_i} [s_j, {\d t}_k] + [f_j \d t_j, \sum^p_{k=1} \frac{\d f_k}{\d x_i} {\d t}_k] = \frac{\d f_j}{\d x_i} {\d t}_j - [\d_i, f_j] \d t_j  = 0.\\
\intertext{The only nonzero commutator arises from }
[t_k - f_k, s_j +  f_j \d t_j] =  [t_k,s_j] + f_j [t_k, \d t_j] - [f_k,s_j] -[f_k, f_j \d t_j] = \delta_{jk} (t_k -f_k).
\end{gather*}
However, according to the definition, only one of these elements belongs to $M_\Lambda$ for any $\Lambda$.

\ref{thm.b} We run Buchberger's algorithm by hands. Due to the ordering property, for each pair the generalized Product Criterion is applicable. Hence using \ref{thm.a} we see, that most $s$-polynomials reduce to commutators, which are zero except $\spoly(t_k - f_k, s_j +  f_j \d t_j) = \delta_{jk} (t_k -f_k)$, which reduces to zero modulo the first polynomial. Thus, any subset including $M_\Lambda$ is indeed a Gr\"obner basis.

\ref{thm.c} Using pairwise commutativity, we employ the Commutative Preimage Theorem from \cite{Lev05}. It states, that the ideal of algebraic dependencies between pairwise commuting elements $\{h_k\mid 1\leq k \leq m\}\subset E$ can be computed as
\[
E \otimes_{\K} \K[c_1,\ldots,c_m] \supset \langle \{ h_i - c_i \} \rangle \cap \K[c_1,\ldots,c_m],
\]
where $c_i$ are new commutative variables, adjoint to $E$. In this elimination problem one requires an ordering on $E \otimes_{\K} \K[c]$, preferring variables of $E$ to $c_i$'s. For such an ordering, one needs to compute a Gr\"obner basis. Now, take $\{h_i\}:=M_\Lambda$, $1\leq i \leq p+n$, and run Buchberger's algorithm with respect to the same ordering as in \ref{thm.b}. Thus we are again in the situation, where the Product Criterion applies, hence $[h_i - c_i, h_k - c_k] =  0$ since $[h_i, h_k]=0$ by \ref{thm.b} and $c_i$ are central. Hence, $\{ h_i - c_i \}$ is a left Gr\"obner basis and by the elimination property $\langle \{ h_i - c_i \} \rangle \cap \K[c_1,\ldots,c_m] = 0$, that is $\{h_i\}$ are algebraically independent.

\ref{thm.d} By \ref{thm.c}, $M_\Lambda$ generates a commutative ring with no algebraic dependence between its elements, so the Krull dimension is the cardinality of $M_\Lambda$, that is $n+p$. Since $M_\Lambda$ is isomorphic to a commutative polynomial ring by \ref{thm.c}, its Gel'fand-Kirillov dimension over the field $\K$ is $n+p$ as well.

\ref{thm.e} With respect to the ordering from \ref{thm.b}, the leading monomials of the generators are $\{\d_1,\ldots,\d_n\} \cup \{ t_k \mid k\in\Lambda \} \cup \{s_j \mid j\not\in\Lambda \}$. Assume, that there exists an element in $E\setminus \K[M_\Lambda]$, which commutes with all elements in $M_\Lambda$. Then its leading monomial must belong to the subalgebra $F$, generated by $\{x_1,\ldots,x_n\} \cup \{ s_k \mid k\in\Lambda \} \cup \{t_j \mid j\not\in\Lambda \} \cup \{\d t_1,\ldots,\d t_p\}$. Since the center of $E$ is $\K$, we consider centralizers of elements. Taking $F'= \cap_{k\in\Lambda } C(t_k-f_k) \cap F$, we see that an element from it can have no $\{\d t_k, s_k \mid k\in\Lambda \}$. Considering $F''=\cap_i C(g_i) \cap F' $, we exclude $\{x_1,\ldots,x_n\}$. Thus we are left with the subalgebra, generated by $\tilde{F} = \{\d t_j, s_j \mid j\not\in\Lambda \}$. But no element of it can commute with $\{ s_j + f_j \d t_j \mid j\not\in\Lambda \}$ except constants. Hence the claim.
\end{proof}

We want to eliminate both $\{t_j\}$ and $\{\d t_j\}$ from an ideal, generated by $G \cup S \cup T$. By using an elimination ordering for $\{t_j\}$ we proved in \ref{thm.b} above, that $G \cup S \cup T$ is a Gr\"obner basis. Hence, the elimination ideal is generated by $G \cup S$ and we can proceed with eliminating $\{\d t_j\}$ from $G \cup S$, which is exactly the statement of \BM in Section \ref{BM}.

\subsection{Implementation}

We use the following acronyms in addressing functions in the implementation: \textit{OT} for Oaku and Takayama, \textit{LOT} for Levandovskyy's modification of Oaku and Takayama \cite{LM08} and \textit{BM} for \BM. Moreover, it is possible to specify the desired Gr\"obner basis engine (\texttt{std} or \texttt{slimgb}) via an optional argument. 

For the classical situation $f = f_1$, the procedure \texttt{Sannfs($f$)} computing $\Ann_{D_n[s]}(f^s) \subset D_n[s]$ uses a \doublequote{minimal user knowledge} principle and chooses one of three mentioned algorithms. Alternatively, one can call the corresponding procedures \texttt{SannfsOT, SannfsLOT, SannfsBM} directly.

For the annihilator of $f = f_1 \cdots f_p$, see Section \ref{BMI}.

\begin{example}\label{exSannfs}
We demonstrate, how to compute the $s$-parametric annihilator with \texttt{Sannfs}. This procedure takes a polynomial in a commutative ring as its argument and returns back a Weyl algebra of the type \texttt{ring} containing an object of the type \texttt{ideal} called \texttt{LD}. Note, that the latter ideal is a set of generators and not a Gr\"obner basis in general.

\begin{footnotesize}\begin{verbatim}
LIB "dmod.lib";
ring r = 0,(x,y),dp;          // set up the commutative ring
poly f = x^3 + y^2 + x*y^2;   // define the polynomial 
def D = Sannfs(f); setring D; // call Sannfs and change to ring D
LD = groebner(LD); LD;        // compute and print Groebner basis
==> LD[1]=2*x*y*Dx-3*x^2*Dy-y^2*Dy+2*y*Dx
==> LD[2]=2*x^2*Dx+2*x*y*Dy+2*x*Dx+3*y*Dy-6*x*s-6*s
==> LD[3]=x^2*y*Dy+y^3*Dy-2*x^2*Dx-3*x*y*Dy-2*y^2*s+6*x*s
==> LD[4]=x^3*Dy+x*y^2*Dy+y^2*Dy-2*x*y*s-2*y*s
==> LD[5]=2*y^3*Dx*Dy+3*x^3*Dy^2+x*y^2*Dy^2-4*x^2*Dx^2-8*x*y*Dx*Dy-2*x^2*Dx
 -4*y^2*Dx*s+6*x*y*Dy+12*x*Dx*s-10*x*Dx-6*y*Dy+12*s
\end{verbatim}\end{footnotesize}
\end{example}

\section{Annihilators of polynomial and rational functions}\label{anniA}

\subsection{Annihilator of $f^{\alpha}$ for $\alpha \in \C$}\label{annfA}

It is known (\eg \cite{SST00}) that for any $\alpha \in \C$, $D_n/\ann_{D_n}(f^{\alpha})$ is a holonomic $D$-module. In the procedure \texttt{annfspecial} from \texttt{dmod.lib} we follow Algorithm 5.3.15 in \cite{SST00}. Given $f$ and $\alpha$, we compute $\Ann_{D_n[s]}(f^s) \subset D_n[s]$, the \BS polynomial of $f$ (cf. Section \ref{globalBS}) and its minimal integer root $s_0$. Then, if $\alpha- (s_0 + 1) \in \N$, according to Algorithm 5.3.15 in \cite{SST00} we have to compute a certain syzygy module in advance. 
Otherwise, $\ann_{D_n}(f^{\alpha}) = \Ann_{D_n[s]}(f^s) \mid_{s=\alpha}$ is obtained via substitution.

\begin{example}\label{exAnnSpecial}
In this example we show, how one computes the annihilator of $2xy$.

\begin{footnotesize}\begin{verbatim}
LIB "dmod.lib"; option(redSB); option(redTail); 
ring r = 0,(x,y),dp; poly g = 2*x*y;  
def A = Sannfs(g); setring A;        // compute Ann(g^s)
LD = groebner(LD); LD;               // GB of the ideal Ann(g^s)
==> LD[1]=y*Dy-s
==> LD[2]=x*Dx-s
def B = annfs0(LD,2*x*y); setring B; // compute BS polynomial
BS; // the list of roots and multiplicities of BS polynomial
==> [1]:
==>    _[1]=-1
==> [2]:
==>    2
// so, the minimal integer root is -1
setring A;                           // need to work with Ann(g^s) again
ideal I = annfspecial(LD,2*x*y,-1,1); 
                       // the last argument 1 indicates that we want to compute f^1
print(matrix(I));                    // condensed presentation
==> Dy^2, y*Dy-1, Dx^2, x*Dx-1
\end{verbatim}\end{footnotesize}
\end{example}

\subsection{Alternative for an annihilator of $f^m$}

Computing a syzygy module in the previous algorithm can be expensive. Therefore we note, that for $\alpha=m \in \N$ we better use an easier approach. 

\begin{lemma}\label{annPoly}
Let $g\in \K[x_1,\ldots,x_n]$. Consider the homomorphism of left $D_n$-modules $\psi: D_n \rightarrow D_n/\langle \d_1,\ldots,\d_n\rangle$, $\psi(1) = g$. Then $\ann_{D_n}(g) = \ker \psi$.
\end{lemma}

\begin{proof}
Note, that $\K[x_1,\ldots,x_n] \cong D_n/\langle \d_1,\ldots,\d_n\rangle$ as left $D_n$-modules. Hence we can view $g$ as the image of $1$ under $\psi$. Then $\ann_{D_n}(g) = \{a\in D_n \mid a\bullet g = 0\} = \{a\in D_n \mid ag \in \langle \d_1,\ldots,\d_n\rangle \} = \ker \psi$. 
\end{proof}

\begin{remark}
Hence, given any element $f\in\K[x_1,\ldots,x_n]$, $\ann_{D_n}(f)$ can be computed via the kernel of a module homomorphism (algorithm \texttt{Modulo}) which amounts to just one Gr\"obner basis computation. Moreover, it does not use elimination and hence is clearly more efficient in the special case $g=f^n$ for $f\in R, n\in\N$, than the more general method in Section \ref{annfA}. Notably this method can be generalized to various other operator algebras, see \cite{SLZ09} for details. The corresponding procedure in \texttt{dmodapp.lib} is called \texttt{annPoly}.
\end{remark}

\begin{remark}\label{annPolyCheckRoot}
Yet another improvement can be achieved in the computation of the minimal integer root of the \BS polynomial with the algorithms from Theorem \ref{checkRoot} below. Namely, since we know, that for an integer root, say $\alpha$, of the \BS polynomial of a polynomial in $n \geq 2$ variables  $-n+1 \leq \alpha \leq -1$ holds (by \cite{Saito94, Varcenko81}) and $-1$ is always a root, we can run the \texttt{checkRoot} procedure (which is just one Gr\"obner basis computation with an arbitrary ordering, see Section \ref{checkrootA}) starting from $\alpha = -n+1$ to $\alpha=-2$. We stop at the first affirmative answer from \texttt{checkRoot} or output $-1$ if no positive answer appears. Thus, one executes \texttt{checkRoot} at most $n-2$ times.
\end{remark}

\begin{algorithm}[Heuristic for $\Ann_{D_n}(f^{\alpha})$]\label{heurAnnFa}~
\begin{algorithmic}
\REQUIRE $f\in\C[x_1,\ldots,x_n]$, $\alpha\in\C$
\ENSURE $\Ann_{D_n}(f^{\alpha})$
\STATE \textbf{if} $\alpha \in \C \setminus (\Z \cap [-n+1,-1])$ \textbf{then}
\STATE \qquad
$\Ann_{D_n} (f^{\alpha}) = 
\begin{cases}
\langle \d_1,\ldots,\d_n\rangle & \text{ if } \alpha=0, \\
\ker(D_n \overset{1\mapsto f^{m}}{\longrightarrow} D_n/\langle \d_1,\ldots,\d_n\rangle) & \text{ if } \alpha=m\in\N, \quad \text{(cf. \ \ref{annPoly})},\\
\Ann_{D_n[s]}(f^{s}) \mid_{s=\alpha} & \text{ if } \alpha\in (\C\setminus \Z)  \cup (\Z\cap (-\infty,-n]),
\end{cases}$\\
\textbf{else} (that is $\alpha \in \Z \cap [-n+1,-1]$)
\STATE \qquad $\mu := \min\{\beta \in \Z_{<0} \mid b_f(\beta)=0\}$
\STATE \qquad
$\Ann_{D_n} (f^{\alpha}) = 
\begin{cases}
\text{Procedure \ref{annfA} with \ref{annPolyCheckRoot}} & \text{ if } 
\mu +1 \leq \alpha \leq -1, \\
\text{Procedure \ref{annPolyCheckRoot} and} \Ann_{D_n[s]} (f^{s}) \mid_{s=\alpha} & 
\text{ if } -n+1 \leq \alpha \leq \mu. \\
\end{cases}$
\STATE \textbf{end if}
\RETURN $\Ann_{D_n} (f^{\alpha})$
\end{algorithmic}
\end{algorithm}

\subsection{Annihilator of a rational function}

In order to compute the annihilator $I$ of a rational function $\frac{f}{g}$ (it is known that $D_n/I$ is holonomic) we use the following lemma.

\begin{lemma}\label{annRat}
Let $f, g\in \K[x_1,\ldots,x_n]\setminus\{0\}$. Consider the homomorphism of left $D_n$-modules $\tau: D_n \rightarrow D_n/\ann_{D_n}(g^{-1}), q \mapsto qf$. Then $\ann_{D_n}(\frac{f}{g}) = \ker\tau$.
\end{lemma}

\begin{proof}
For $q\in\ker\tau=\{ q\in D_n\mid qf \in \ann_{D_n}(g^{-1})\}$, $(qf) \bullet g^{-1} = q\bullet (fg^{-1})$, hence $\ann_{D_n}(\frac{f}{g}) = \ker\tau$.
\end{proof}

We compute $\ann_{D_n}(g^{-1})$ with Algorithm \ref{heurAnnFa} above. Although in the case, when $-1$ is not the minimal integer root of the \BS polynomial of $g$, we have to use expensive algorithms like \ref{annfA}, we know no other methods to compute the annihilator in Weyl algebras. Also, no general algorithm for computing a complete system of operator equations (with operators including along partial differentiation also partial ($q$-)differences et cetera) with polynomial coefficients, annihilating a rational function, is known to us. In our opinion, the existence of an algorithm for $\ann_{D_n}(g^{-1})$ shows the intrinsic naturality of $D$-modules compared with other linear operators acting on $\K[x]$. The algorithm is implemented in \texttt{dmodapp.lib} and the corresponding procedure is called \texttt{annRat}.

\begin{example}\label{exAnnRat}
In this example we demonstrate the computation of annihilators of a rational function. The procedure \texttt{annRat} takes as arguments polynomials in a commutative ring and returns a Weyl algebra (of type \texttt{ring}) together with an object of type \texttt{ideal} called \texttt{LD} (cf. Example \ref{exSannfs}). Note, that \texttt{LD} is given in a Gr\"obner basis.

\begin{footnotesize}\begin{verbatim}
LIB "dmodapp.lib";
ring r = 0,(x,y),dp;
poly g = 2*x*y;  poly f = x^2 - y^3; // we will compute Ann(g/f)
option(redSB); option(redTail);      // get reduced minimal GB
def B = annRat(g,f); setring B;
LD;                                  // Groebner basis of Ann(g/f)
==> LD[1]=3*x*Dx+2*y*Dy+1
==> LD[2]=y^3*Dy^2-x^2*Dy^2+6*y^2*Dy+6*y
==> LD[3]=9*y^2*Dx^2*Dy-4*y*Dy^3+27*y*Dx^2+2*Dy^2
==> LD[4]=y^4*Dy-x^2*y*Dy+2*y^3+x^2
==> LD[5]=9*y^3*Dx^2-4*y^2*Dy^2+10*y*Dy-10
\end{verbatim}\end{footnotesize}
\end{example}

\section{$b$-function with respect to weights for an ideal}\label{bfctIdeal}

Let $0 \neq w \in \R^n_{\geq 0}$ and consider the $V$-filtration $V = \left\{ V_m \mid m \in \Z \right\}$ on $D_n$ with respect to $w$,  where $V_m$ is spanned by $\left\{ x^{\alpha} \d^{\beta} \mid -w \alpha + w \beta \leq m \right\}$ over $\K$. That is, $x_i$ and $\d_i$ get weights $-w_i$ and $w_i$ respectively. Note that then the relation $\d_i x_i = x_i \d_i + 1$ is homogeneous of degree $0$. It is known that the associated graded ring $\bigoplus_{m \in \Z} V_m / V_{m-1}$ is isomorphic to $D_n$, which allows us to identify it with the Weyl algebra.

From now on we assume, that $I \subset D_n$ is an ideal such that $D_n/I$ is a holonomic module. Since holonomic $D$-modules are cyclic (\eg \cite{Cou95}), for each holonomic $D$-module $M$ there exists an ideal $I_M$ such that $M \cong D_n/I_M$ as $D$-modules.

\begin{definition}
Let $0 \neq w \in \R^n_{\geq 0}$. For a non-zero polynomial 
\[
p = \sum_{\alpha, \beta \in \N_0^n} c_{\alpha \beta} x^{\alpha} \d^{\beta} \in D_n
\quad \text{with all but finitely many } c_{\alpha \beta} = 0
\]
we put $m = \max_{\alpha, \beta} \{ -w \alpha + w \beta \mid c_{\alpha \beta} \neq 0\} \in \R$ and define the \emph{initial form} of $p$ with respect to the weight $w$ as follows:
\[
\inw(p) := \sum_{\alpha,\beta \in \N_0^n:~ -w \alpha + w \beta = m} c_{\alpha \beta} x^{\alpha} \d^{\beta}.
\]
For the zero polynomial, we set $\inw(0) := 0$. Additionally, the ideal $\inw(I) := \K \cdot \{ \inw(p) \mid p \in I \}$ is called the \emph{initial ideal} of $I$ with respect to $w$.
\end{definition}

\begin{definition}
Let $0 \neq w \in \R^n_{\geq 0}$ and $s := \sum_{i=1}^n w_i x_i \d_i$. Then $\inw(I) \cap \K[s]$ is a principal ideal in $\K[s]$. Its monic generator $b_{I,w}(s)$ is called the \emph{global $b$-function} of $I$ with respect to the weight $w$.
\end{definition}

\begin{theorem}\label{b(s) is not zero}
The global $b$-function is nonzero.
\end{theorem}

We will give a proof of this well-known result in Section \ref{intersection}.

Following its definition, the computation of the global $b$-function of $I$ with respect to~$w$ can be done in two steps:
\begin{enumerate}
\item Compute the initial ideal $I'$ of $I$ with respect to $w$.
\item Compute the intersection of $I'$ with the subalgebra $\K[s]$.
\end{enumerate}

We will discuss both steps separately, starting with the initial ideal. It is important to mention, that although this procedure has been described in \cite{SST00}, this approach was completely treated by Noro in \cite{Noro02}, accompanied with a very impressive implementation in \texttt{Risa/Asir}. 

\subsection{Computing the initial ideal} \label{initial ideal}

In order to compute the initial ideal, the method of weighted homogenization is proposed in \cite{Noro02}, which we will describe below.

Let $u,v \in \R^n_{>0}$. The $G$-algebra $D_{(u,v)}^{(h)} := \K \langle x_1,\ldots,x_n,\d_1,\ldots,\d_n, h \mid \{x_j x_i = x_i x_j, \d_j \d_i \\= \d_i \d_j, x_i h = h x_i, \d_i h = h \d_i, \d_j x_i = x_i \d_j + \delta_{i,j} h^{u_i + v_j} \} \rangle$ is called the \emph{$n$-th weighted homogenized Weyl algebra} with weights $u,v$, \ie $x_i$ and $\d_i$ get weights $u_i$ and $v_i$ respectively.

For $p = \sum_{\alpha, \beta} c_{\alpha \beta} x^{\alpha} \partial^{\beta} \in D_n$ we define the \emph{weighted homogenization} of $p$ as follows:
\[
H_{(u,v)}(p) = \sum_{\alpha, \beta} c_{\alpha \beta} h^{\deg_{(u,v)}(p) - (u \alpha + v \beta)} x^{\alpha} \d^{\beta}.
\]
This definition naturally extends to a set of polynomials. Here, $\deg_{(u,v)}(p)$ denotes the weighted total degree of $p$ with respect to weights $u,v$ for $x,\d$ and weight $1$ for $h$.

For a monomial ordering $\prec$ on $D_n$, which is not necessarily a well-ordering, we define an associated homogenized global ordering $\prec^{(h)}$ in $D_{(u,v)}^{(h)}$ by setting $h \prec^{(h)} x_i, h \prec^{(h)} \d_i$ for all $i$ and,
\begin{alignat*}{2}
p \prec^{(h)} q \quad
&\text{if} \quad &\deg_{(u,v)}(p) &< \deg_{(u,v)}(q)\\
&\text{or} &\deg_{(u,v)}(p) &= \deg_{(u,v)}(q) 
\quad \text{and} \quad 
p_{\mid_{h=1}} \prec q_{\mid_{h=1}}.
\end{alignat*}

Note that for $u = v = (1,\ldots,1)$ this is exactly the standard homogenization as in \cite{SST00} and \cite{CN97}. Analogue statements of the following two theorems can be found in \cite{SST00} and \cite{Noro02} respectively.

\begin{theorem}\label{Theorem: dehomogenization}
Let $F$ be a finite subset of $D_n$ and $\prec$ a global ordering. If $G^{(h)}$ is a Gr\"obner basis of $\langle H_{(u,v)}(F) \rangle$ with respect to $\prec^{(h)}$, then ${G^{(h)}}_{\mid_{h=1}}$ is a Gr\"obner basis of $\langle F \rangle$ with respect to $\prec$.
\end{theorem}

\begin{theorem}\label{Theorem: initial ideal}
Let $\prec$ be a global monomial ordering on $D_n$ and $\prec_{(-w,w)}$ the non-global ordering defined by
\begin{alignat*}{2}
		x^{\alpha} \d^{\beta} \prec_{(-w,w)} x^{\gamma} \d^{\delta} \quad
		&\text{if}\quad  & -w \alpha + w \beta &< -w \gamma + w \delta\\
		&\text{or} &-w \alpha + w \beta &= -w \gamma + w \delta
			\quad\text{and}\quad x^{\alpha} \d^{\beta} \prec x^{\gamma} \d^{\delta}.
\end{alignat*}
If $G^{(h)}$ is a Gr\"obner basis of $H_{(u,v)}(I)$ with respect to $\prec^{(h)}_{(-w,w)}$, then the set $\{ \inwh(g) \mid g \in G^{(h)}\}$ is a Gr\"obner basis of $\inwh(H_{(u,v)}(I))$ with respect to $\prec^{(h)}$.
\end{theorem}

\begin{proof}
Let $f' \in \inwh(H_{(u,v)}(I))$ be $(-w,w,0)$-homogeneous. There exist $f \in H_{(u,v)}(I)$ and $g \in G^{(h)}$ such that $f' = \inwh(f)$ and $\lm_{\prec^{(h)}_{(-w,w)}}(g) \mid \lm_{\prec^{(h)}_{(-w,w)}}(f)$. Since $f,g$ are $(u,v)$-ho\-mo\-ge\-neous, we have
$$
\lm_{\prec^{(h)}_{(-w,w)}}(g) = \lm_{\prec^{(h)}}(\inwh(g)),\qquad
\lm_{\prec^{(h)}_{(-w,w)}}(f) = \lm_{\prec^{(h)}}(\inwh(f)),
$$
which concludes the proof.
\end{proof}

Summarizing the results from this section, we obtain the following algorithm to compute the initial ideal.

\begin{algorithm}[\texttt{InitialIdeal}]\label{InitialIdeal}
\begin{algorithmic}
	\REQUIRE $I\subset D_n$ a holonomic ideal, 
		$0 \neq w \in \R^n_{\geq 0}$, 
		$\prec$ a global ordering on $D_n$, 
		$u,v \in \R^n_{>0}$
	\ENSURE A Gr\"obner basis $G$ of $\inw(I)$ with respect to $\prec$
	\STATE $\prec^{(h)}_{(-w,w)} :=$ the homogenized ordering as defined in theorem \ref{Theorem: initial ideal}
	\STATE $G^{(h)} :=$ a Gr\"obner basis of $H_{(u,v)}(I)$ with respect to $\prec^{(h)}_{(-w,w)}$
	\RETURN $G = \inw({G^{(h)}}_{\mid_{h=1}})$
\end{algorithmic}
\end{algorithm}

\subsection{Intersecting an ideal with a principal subalgebra}\label{intersection}

We will now consider a much more general setting than needed to compute the global $b$-function. Let $A$ be an associative $\K$-algebra. We are interested in computing the intersection of a left ideal $J \subset A$ with the subalgebra $\K[s]$ of $A$ where $s \in A$ is an arbitrary non-constant element. This intersection is always generated by one element since $\K[s]$ is a principal ideal domain. In other words, we want to find the monic generator $b \in A$ such that \ $\langle b \rangle = J \cap \K[s]$.

\medskip
For this section, we will assume that there is an ordering on $A$ such that there exists a finite left Gr\"obner basis $G$ of $J$.

\medskip
Then we can distinguish between the following four situations:
\begin{enumerate}[label=\arabic{*}., ref=\arabic{*}.]
	\item \label{sit 1} No leading monomials of elements in $G$ divide the leading monomial of any power of $s$.
	\item \label{sit 2} There is an element in $G$ whose leading monomial divides the leading monomial of some power of $s$.
		In this situation, we have the following sub-situations.
		\begin{enumerate}[label=2.\arabic{*}., ref=2.\arabic{*}.]
			\item \label{sit 2.1} $J \cdot s \subset J$ and $\dim_{\K}(\mathrm{End}_{A}(A/J)) < \infty$.
			\item \label{sit 2.2} One of the two conditions in 2.1. does not hold.
			\begin{enumerate}[label=2.2.\arabic{*}., ref=2.2.\arabic{*}.]
				\item \label{sit 2.2.1} The intersection is zero.
				\item \label{sit 2.2.2} The intersection is not zero.
			\end{enumerate}
		\end{enumerate}
\end{enumerate}

\begin{lemma} \label{la:zero intersection}
If there exists no $g \in G$ such that $\lm(g)$ divides $\lm(s^k)$ for some $k \in \N_0$, then $J \cap \K[s] = \{ 0 \}$.
\end{lemma}

The lemma covers the first case above. In the second case however, we cannot in general state whether the intersection is trivial or not as the following example illustrates.

\begin{remark}
The converse of the previous lemma is wrong. For instance, consider $\K[x,y]$ and $J = \langle y^2 + x \rangle$. Then $J \cap \K[y] = \{ 0 \}$ while $\{ y^2 + x\}$ is a Gr\"obner basis of $J$ for any ordering.
\end{remark}

In situation \ref{sit 2.1} though, the intersection is not zero as the following lemma shows, inspired by the sketch of the proof of Theorem \ref{b(s) is not zero} in \cite{SST00}.

\begin{lemma}\label{lemma minimal polynomial}
Let $J \cdot s \subset J$ and $\dim_{\K}(\mathrm{End}_{A}(A/J)) < \infty$. Then $J \cap \K[s] \neq \{ 0 \}$.
\end{lemma}

\begin{proof}
Consider the right multiplication with $s$ as a map $A/J \ra A/J$ which is a well-defined $A$-module endomorphism of $A/J$ as $a-a' \in J$ implies that $(a-a')s \in J \cdot s \subset J$, which holds by assumption for all $a,a' \in A$. Since $\mathrm{End}_{A}(A/J)$ is finite dimensional, linear algebra guarantees that this endomorphism has a well-defined non-zero minimal polynomial $\mu$. Moreover, $\mu$ is precisely the monic generator of $J \cap \K[s]$ as $\mu(s) = [0]$ in $A/J$, hence $\mu(s) \in J \cap \K[s]$, and $\deg(\mu)$ is minimal by definition.
\end{proof}

\begin{remark}\label{End(A/J)}
In particular, the lemma holds if $A/J$ itself is a finite dimensional $A$-module. In the case where $A$ is a Weyl algebra and $A/J$ is a holonomic module, we know that $\dim_{\K}(\mathrm{End}_{A}(A/J))$ is finite (cf. \cite{SST00}).
\end{remark}

For situation \ref{sit 2.1}, we have reduced our problem of intersecting an ideal with a subalgebra generated by one element to a problem from linear algebra by the proof of the lemma, namely to the one of finding the minimal polynomial of an endomorphism.

\begin{proof}[Proof of Theorem \ref{b(s) is not zero}]
Let $0 \neq w \in \R^n_{\geq 0}, J := \inw(I)$ for a holonomic ideal $I \subset D_n$ and $s := \sum_{i=1}^n w_i x_i \d_i$. Without loss of generality let $0 \neq p = \sum_{\alpha,\beta} c_{\alpha,\beta} x^{\alpha} \d^{\beta} \in J$ be $(-w,w)$-homogeneous. Then we obtain for every monomial in $p$ by using the Leibniz rule
\[
x^{\alpha} \d^{\beta} x_i \d_i = x^{\alpha+e_i} \d^{\beta+e_i} + \beta_i x^{\alpha} \d^{\beta}
= (\d_i x_i^{\alpha_i+1} - (\alpha_i + 1) x_i^{\alpha_i}) \frac{x^{\alpha}}{x_i^{\alpha_i}} \d^{\beta} 
	+ \beta_i x^{\alpha} \d^{\beta} 
\]
\[
= (\d_i x_i - (\alpha_i + 1) + \beta_i) x^{\alpha} \d^{\beta}
= (x_i \d_i - \alpha_i + \beta_i) x^{\alpha} \d^{\beta}.
\]
Put $m = -w \alpha + w \beta$ for some term $c_{\alpha,\beta} x^{\alpha} \d^{\beta}$ in $p$ where $c_{\alpha,\beta}$ is non-zero. Since $p$ is $(-w,w)$-homogeneous, $m$ does not depend on the choice of this term. Hence,
\begin{align*}
p \cdot s
&= p \sum_{i=1}^n w_i x_i \d_i
 = \sum_{i=1}^n w_i \sum_{\alpha,\beta} (x_i \d_i - \alpha_i + \beta_i) c_{\alpha,\beta} x^{\alpha} \d^{\beta} \\
&= s \cdot p + \sum_{i=1}^n \sum_{\alpha,\beta} w_i (- \alpha_i + \beta_i) c_{\alpha,\beta} x^{\alpha} \d^{\beta} 
 = (s + m) \cdot p \in J.
\end{align*}
Since $D_n/J$ is holonomic (cf. \cite{SST00}) and $J \cdot s \subset J$, Remark \ref{End(A/J)} and Lemma \ref{lemma minimal polynomial} yield the claim.
\end{proof}

\medskip
If one knows in advance that the intersection is not zero, the following algorithm terminates.

\begin{algorithm}[\texttt{principalIntersect}] \label{PrincipalIntersect} ~
\begin{algorithmic}
	\REQUIRE $s \in A, J \subset A$ a left ideal such that $J \cap \K[s] \neq \{ 0 \}$.
	\ENSURE $b \in \K[s]$ monic such that $J \cap \K[s] = \langle b \rangle$
	\STATE $G :=$ a finite left Gr\"obner basis of $J$ (assume it exists)
	\STATE $i := 1$
	\LOOP
		\IF{there exist $a_0, \ldots, a_{i-1} \in \K$ such that $\NF(s^i,G) + \sum_{j=0}^{i-1} a_j \NF(s^j,G) = 0$}
  	 	\RETURN $b := s^i + \sum_{j=0}^{i-1} a_j s^j$
  	\ELSE \STATE $i := i+1$
 		\ENDIF
	\ENDLOOP
\end{algorithmic}
\end{algorithm}

Note that because $\NF(s^i,G) + \sum_{j=0}^{i-1} a_j \NF(s^j,G) = 0$ is equivalent to $s^i + \sum_{j=0}^{i-1} a_j s^j \in J$, the algorithm searches for a monic polynomial in $\K[s]$ that also lies in $J$. This is done by going degree by degree through the powers of $s$ until there is a linear dependency. This approach also ensures the minimality of the degree of the output. The algorithm terminates if and only if $J \cap \K[s] \neq \{ 0 \}$. Note that this approach works over any field.

The check whether there is a linear dependency over $\K$ between the computed normal forms of the powers of $s$ is done by the procedure \texttt{linReduce} in our implementation.

\subsubsection{An enhanced computation of normal forms}

When computing normal forms of the form $\NF(s^i,J)$ like in algorithm \ref{PrincipalIntersect} we can speed up the reduction process by making use of the previously computed normal forms.

\begin{lemma} \label{NF computing 1}
Let $A$ be a $\K$-algebra, $J \subset A$ a left ideal and let $f \in A$.
For $i \in \N$ put $r_i = \NF(f^i,J)$, $q_i = f^i - r_i \in J$ and $c_i = \frac{\lc(q_i r_1)}{\lc(r_1 q_i)}$ provided $r_1 q_i \not=0$. For $r_1 q_i=0$ we put $c_i=0$.
Then we have for all $i \in \N$
\[
	r_{i+1} 
	= \NF(f r_i,J)
	= \NF([f^i-r_i,r_1]_{c_i} + r_i r_1,J).
\]
\end{lemma}

As a consequence, we obtain the following result for some $\K$-algebras of special importance.

\begin{corollary} \label{NF computing 2}
If $A$ is a $G$-algebra of Lie type (\eg a Weyl algebra),  then
\[
r_{i+1} = \NF(f r_i,J) = \NF([f^i-r_i,r_1] + r_i r_1,J) \text{  holds}.
\]
If $A$ is commutative, we have $ r_{i+1} = \NF(r_i r_1,J) = \NF(r_1,J)^{i+1}= \NF(r_1^{i+1},J)$.
\end{corollary}

Note, that computing Lie bracket $[f,g]$ both in theory and in practice is easier and faster, than to compute $[f,g]$ as $f\cdot g - g \cdot f$, see \eg \cite{LS03}.

\subsubsection{Applications}

Apart from computing global $b$-functions, there are various other applications of Algorithm~\ref{PrincipalIntersect}.

\paragraph{Solving Zero-dimensional Systems.}

Recall that an ideal $I \subset \K[x_1,\ldots,x_n]$ is called zero-di\-men\-sion\-al if one of the following equivalent conditions holds:
\begin{itemize}
	\item $\K[x_1,\ldots,x_n] / I$ is finite dimensional as a $\K$-vector space.
	\item For each $1 \leq i \leq n$ there exist $0 \neq f_i \in I \cap \K[x_i]$.
	\item The cardinality of the zero-set of $I$ is finite.
\end{itemize}

In order to compute the zero-set of $I$, one can use the classical triangularization algorithms. These algorithms require to compute a Gr\"obner basis with respect to some elimination ordering (like lexicographic one), which might be very hard.

By Algorithm \ref{PrincipalIntersect}, a generator of $I \cap \K[x_i]$ can be computed without these expensive orderings. Instead, any ordering, hence a better suited one, may be freely chosen.

A similar approach is used in the celebrated FGLM algorithm (cf. \cite{FGLM93}).

\paragraph{Computing Central Characters and Algebraic Dependence.}

Let $A$ be an associative $\K$-algebra. Intersection of a left ideal with the center of $A$, which is isomorphic to a commutative ring, is important for many algorithms, among other for the computation of central character decomposition of a finitely presented module (cf. \cite{Lev05} for the theory and \cite{ALM09} for an example with Principal Intersection). In the situation, where the center of $A$ is generated by one element (which is not seldom), we can apply Algorithm \ref{PrincipalIntersect} to compute the intersection (known to be often quite nontrivial) without engaging much more expensive Gr\"obner basis computation, which use elimination.

\begin{example}
Consider the quantum algebra $U'_q(\mathfrak{so}_3)$ (as defined by Fairlie and Odesskii) for $q^2$ being the $n$-th root of unity. It is known, that then, in addition to the single generator $C$ of the center present over any field, three new elements $Z_i$, depending on $n$ will appear. Since $U'_q(\mathfrak{so}_3)$ has Gel'fand-Kirillov dimension $3$, four commuting elements in it obey a single polynomial algebraic dependency (the ideal of dependencies in principal). Computing such a dependency is a very tough challenge for Gr\"obner bases. But as we see, it is quite natural to apply Principal Intersection.

\begin{footnotesize}\begin{verbatim}
LIB "ncalg.lib"; LIB "bfun.lib";
def A = makeQso3(5);             // below Q^2 is the 5th root of unity
setring A; 
// central elements, depending on Q in their classical form: 
ideal I = x5+(Q3-Q2+2)*x3+(Q3-Q2+1)*x, y5+(Q3-Q2+2)*y3+(Q3-Q2+1)*y, 
 z5+(Q3-Q2+2)*z3+(Q3-Q2+1)*z;
I = twostd(I);                   // two-sided Groebner basis
poly C = 5*xyz+(4Q3-3Q2+2Q-1)*x2+(-Q3+2Q2-3Q+4)*y2+(4Q3-3Q2+2Q-1)*z2;
poly v = vec2poly(pIntersect(C,I),1); // present vector as poly
poly t = subst(v,x,C); t;        // t as a polynomial in C of size 42
==> 3125*x5y5z5+(3125Q3-3125Q2+6250)*x5y5z3+(3125Q3-3125Q2+6250)*x5y3z5+ ...
// present matrix of cofactors of t as an element of I:
matrix T = lift(I,t); 
poly a = 125*(25*I[1]*I[2]*I[3]+(Q3-7Q2+8Q-4)*(I[1]^2+I[2]^2+I[3]^2));
a-t;                             // a expresses t in the subalgebra gen. by I[1..3]
==> 0
// define univariate ring over algebraic extension:
ring r = (0,Q),c,dp; minpoly = Q4-Q3+Q2-Q+1;
poly v = fetch(A,v);             // map v from A to a univariate poly in c
factorize(v);
\end{verbatim}\end{footnotesize}

The latter factorization delivers the final touch to the answer: the algebraic dependency is described by the equation $C^{2} \cdot (C+4q^{3}-3q^{2}-3q+4)\cdot (C+3q^{3}-q^{2}-q+3)^{2} = 3125\cdot Z_1 Z_2 Z_3+125\cdot(q^3-7q^2+8q-4)\cdot (Z_1^2 + Z_2^2 + Z_3^2)$.
\end{example}

\section{Bernstein-Sato polynomial of $f$} \label{bspolyA}

\subsection{Global Bernstein-Sato polynomial}
\label{globalBS}

One possibility to define the Bernstein-Sato polynomial of a polynomial $f \in \K[x_1,\ldots,x_n]$ is to apply the global $b$-function for specific weights.

\begin{definition}
Let $b_{I_f,w}(s)$ denote the global $b$-function of the univariate Malgrange ideal $I_f$ of $f$ (cf. Section \ref{sAnn}) with respect to the weight vector $w = (1,0,\ldots,0) \in \R^{n+1}$, that is the weight of $\d_t$ is $1$. Then $b_f(s) = (-1)^{\deg(b_{I_f,w})} b_{I_f,w}(-s-1)$ is called the \emph{global $b$-function} (\emph{Bernstein-Sato polynomial}) of $f$.
\end{definition}

By Theorem \ref{b(s) is not zero}, $b_f(s) \neq 0$ holds. Moreover, it is known that all roots of $b_f(s)$ are negative rational numbers. Kashiwara proved this result for local Bernstein-Sato polynomials over $\C$ \cite{Kashiwara76/77}. This fact together with Theorem \ref{localBS} below and classical flatness properties imply the claim for the global case over an arbitrary field of characteristic $0$.

The following theorem gives us another option to define the Bernstein-Sato polynomial.

\begin{theorem}[\cite{Bernstein71}, see also {\cite[Lemma 5.3.11]{SST00}}] \label{bfct and ann}
The Bernstein-Sato polynomial $b_f(s)$ of $f$ is the unique monic polynomial of minimal degree in $\K[s]$ satisfying the identity
\[
	P \bullet f^{s+1} = b_f(s) \cdot f^s \qquad \text{for some operator } P \in D_n[s].
\]
\end{theorem}

Since $P \cdot f - b_f(s) \in \Ann_{D_n[s]}(f^s)$ holds, $b_f(s)$ is the monic polynomial satisfying
\begin{equation}\label{bernstein2}
\langle b_f(s) \rangle = \Ann_{D_n[s]}(f^s) + \langle f \rangle \cap \K[s].
\end{equation}

Summarizing, there are several choices for computing the Bernstein-Sato polynomial:
\begin{enumerate}
	\item Compute either\\
		(a) $J = \ini_{(-w,w)}(I_f)$ 
		or\\ 
		(b) $J = \ann_{D_n[s]}(f^s) + \langle f \rangle$.
	\item Intersect $J$ with $\K[s]$ by\\
		(a) the classical elimination-driven approach or\\
		(b) using Algorithm \ref{PrincipalIntersect}.
\end{enumerate}

It is very interesting to investigate the approach for the computation of \BS polynomial, which arises as the combination of the two methods:
\begin{enumerate}
	\item $\Ann_{D_n[s]}(f^s)$ via Brian\c{c}on-Maisonobe (cf. \cite{LM08}),
	\item $(\Ann(_{D_n[s]}f^s) + \langle f \rangle) \cap \K[s]$ via Algorithm \ref{PrincipalIntersect}.
\end{enumerate}

For an efficient computation of $\ini_{(-w,w)}(I_f)$ using the method of weighted homogenization as described in Section \ref{initial ideal}, Noro proposes \cite{Noro02} to choose the weights $\hat{u} = (\deg_u(f), u_1, \ldots, u_n )$, $\hat{v} = (1, \deg_u(f) - u_1 + 1, \ldots, \deg_u(f) - u_n + 1)$, such that the weight of $t$ is $\deg_u(f)$ and the weight of $\d_t$ is $1$. Here, $u \in \R^n_{>0}$ is an arbitrary vector and $\deg_u(f)$ denotes the weighted total degree of $f$ with respect to $u$. The vector $u$ may be chosen heuristically in accordance to the shape of $f$ or by default, one can set $u = (1,\ldots,1)$.

\subsection{Implementation}

For the computation of Bernstein-Sato polynomials, we offer the following procedures in the \textsc{Singular} library \texttt{bfun.lib}:

\texttt{bfct} 
computes $\inw(I_f)$ using weighted homogenization with weights $\hat{u},\hat{v}$ for an optional weight vector $u$ (by default $u = (1,\ldots,1))$ as described above, and then uses Algorithm \ref{PrincipalIntersect}, where the occurring systems of linear equations are solved by the procedure \texttt{linReduce}.

\texttt{bfctAnn}
computes $\Ann_{D_n[s]}(f^s)$ via Brian\c{c}on-Maisonobe and intersects $\Ann_{D_n[s]}(f^s) + \langle f \rangle$ with $\K[s]$ analogously to \texttt{bfct}.

\texttt{bfctOneGB}
computes the initial ideal and the intersection at once using a homogenized elimination ordering, a similar approach has been used in \cite{HH01}.

For the global $b$-function of an ideal $I\subset D_n$, \texttt{bfctIdeal} computes $\inw(I)$ using standard homogenization, \ie weighted homogenization where all weights are equal to $1$, and then proceeds the same way as \texttt{bfct}. Recall that $D_n/I$ must be holonomic as in \cite{SST00}.

All these procedures work as the following example illustrates for \texttt{bfct} and the hyperplane arrangement  $xyz(z-y)(y+z)$.

\begin{footnotesize}\begin{verbatim}
LIB "bfun.lib";
ring r = 0,(x,y,z),dp;    // commutative ring
poly f = x*y*z*(z-y)*(y+z);
list L = bfct(f);
print(matrix(L[1]));      // the roots of the BS-polynomial
==> -1,-5/4,-3/4,-3/2,-1/2
L[2];                     // the multiplicities of the roots above
==>  3,1,1,1,1
\end{verbatim}\end{footnotesize}

\subsection{Local \BS Polynomial}

Here we are interested in what kind of information one can obtain from the local $b$-functions for computing the global one and conversely. In order to avoid theoretical problems we will assume in
this paragraph that the ground field $\K=\C$.

Several algorithms to obtain the local $b$-function of a hypersurface $f$ have been known without any Gr\"obner bases computation but under some conditions on $f$. For instance, it was shown by Malgrange \cite{Malgrange75} that the minimal polynomial of $-\d_t t$ acting on some vector space of finite
dimension coincides with the reduced (local) Bernstein polynomial, assuming that the singularity is isolated.

The algorithms of Oaku \cite{Oaku97} used Gr\"obner bases for the first time. Recently, Nakayama presented some algorithms, which use the global $b$-function as a bound and obtain a local $b$-function by Mora resp. approximate division \cite{Nakayama09}, see also the work of Nishiyama and Noro \cite{NN10}.

\begin{theorem}[Brian\c con-Maisonobe (unpublished), Mebkhout-Narv\'aez \cite{MN98}]\label{localBS}
Let $b_{f,P}(s)$ be the local $b$-function of $f$ at the point $P\in\C^n$ and $b_f(s)$ the global one. Then it is verified that $b_f(s) = \lcm_{P\in\C^n} b_{f,P}(s) = \lcm_{P\in\Sigma(f)} b_{f,P}(s)$, where $\Sigma(f) = V(\langle f,\frac{\d f}{\d x_1}, \ldots, \frac{\d f}{\d x_n} \rangle)$ denotes the singular locus of $V(f)$.
\end{theorem}

\begin{remark}
Assume, that $\Sigma(f)$ consists of finitely many isolated singular points (the dimension of the corresponding defining ideal is $0$). Then the computation of the global $b$-function with Theorem \ref{localBS} becomes effective. Moreover, one needs just an algorithm for computing the local $b$-function of a hypersurface, having an isolated singularity at the origin.
\end{remark}

The {\sc Singular} library {\tt gmssing.lib}, developed and implemented by M. Schulze \cite{Schulze04}, contains the procedure {\tt bernstein}, which computes the local $b$-function at the origin. It returns the list of roots and corresponding multiplicities.

\begin{example}
Let $\mathcal{C}$ be the curve in $\C^2$ given by $f = (x^3-y^2)(3x-2y-1)(x+2y)$. This curve has three isolated singular points $p_1 = (0,0)$, $p_2 = (1,1)$ and $p_3 = (1/4,-1/8)$. 

\begin{footnotesize}\begin{verbatim}
LIB "gmssing.lib";
// note that one must use a local ordering for calling 'bernstein'
ring r = 0,(x,y),ds;   // ds stands for a local degrevlex ordering
poly f = (x^3-y^2)*(3x-2y-1)*(x+2y);
list L = bernstein(f); // local b-function at the origin p_1
print(matrix(L[1]));
==> -11/8,-9/8,-1,-7/8,-5/8
L[2];
==>    1,1,2,1,1
\end{verbatim}\end{footnotesize}

\noindent Moving to the corresponding points we also compute $b_{f,P_2}(s)$ and $b_{f,P_3}(s)$.
\begin{align*}
b_{f,P_1}(s) &= (s+1)^2 (s+5/8) (s+7/8) (s+9/8) (s+11/8)\\
b_{f,P_2}(s) &= (s+1)^2 (s+3/4) (s+5/4)\\
b_{f,P_3}(s) &= (s+1)^2 (s+2/3) (s+4/3)
\end{align*}
From this information and using Theorem \ref{localBS}, the global $b$-function is
\[
b_f(s)=(s+2/3) (s+5/8) (s+3/4) (s+7/8)  (s+1)^2  (s+4/3)  (s+5/4)  (s+9/8) (s+11/8).
\]
\end{example}

Moreover, {\tt gmssing.lib}, allows one to compute invariants related to the Gauss-Manin system of an isolated hypersurface singularity.

In the non-isolated case the situation is more complicated. For computing the local $b$-function in this case (which is important on its own) we suggest using two methods: Take the global $b$-function as an upper bound and a local version of the {\tt checkRoot} algorithm, see below. Another method is to use a local version of \texttt{principalIntersect}, which is under development. Despite the existence of many algorithms, the effectiveness of the computation of local $b$-functions is still to be drastically enhanced.

\subsection{Partial knowledge of \BS polynomial} \label{checkrootA}

As we have mentioned, several algorithms for computing the $b$-function associated with a polynomial have been known. However, in general it is very hard from computational point of view to obtain this polynomial, and in the actual computation a limited number of examples can be treated. For some applications only the integral roots of $b_f(s)$ are needed and that is why we are interested in obtaining just a part of the \BS polynomial.

Recall the algorithm {\tt checkRoot} for checking whether a rational number is a root of the $b$-function of a hypersurface from \cite{LM08}. Equation \eqref{bernstein2} was used to prove the following result.

\begin{theorem}(\cite{LM08})\label{checkRoot}
Let $R$ be a ring whose center contains $\K[s]$ as a subring. Let us consider $q(s)\in \K[s]$ a polynomial in one variable and $I$ a left ideal in $R$ satisfying $I\cap \K[s]\neq 0$. Then
$
  ( I + R \langle q(s) \rangle )\cap \K[s] = I\cap \K[s] + \K[s] \langle q(s) \rangle.
$
In particular, using the above equation \eqref{bernstein2}, we have
$$
  \big(\, \ann_{D_n[s]}(f^s)+ D_n[s]\cdot \langle f, q(s)\rangle \,\big) \; \cap \K[s] = \langle b_f(s),q(s)\rangle.
$$
As a consequence, let $m_\alpha$ be the multiplicity of $\alpha$ as a root of $b_f(-s)$ and let us consider the ideals $J_i = \ann_{D_n[s]}(f^s) + \langle f, (s+\alpha)^{i+1} \rangle \subseteq D_n[s]$, $i=0,\ldots,n$, then $[\, m_\alpha > i \Longleftrightarrow (s+\alpha)^i \notin J_i \,]$.
\end{theorem}

Once we know a system of generators of the annihilator of $f^s$ in $D_n[s]$, the last theorem provides an algorithm for checking whether a given rational number is a root of the $b$-function of $f$ and for computing its multiplicity, using Gr\"obner bases for differential operators.

This algorithm is much faster, than the computation of the whole Bernstein polynomial via Gr\"obner bases, because no elimination ordering is needed for computing a Gr\"obner basis of $J_i$, once one knows a system of generators of $\ann_{D_n[s]}(f^s)$. Also, the element $(s+\alpha)^{i+1}$, added as a generator, seems to simplify tremendously such a computation. Actually, when $i=0$ it is possible to eliminate the variable $s$ in advance and we can perform the whole computation in $D_n$. Let us see an example.

\begin{example}
Let $A$ be the matrix given by
$$
A = \left(\begin{array}{cccc}
x_1 & x_2 & x_3 & x_4\\
x_5 & x_6 & x_7 & x_8\\
x_9 & x_{10} & x_{11} & x_{12}
\end{array}\right).
$$
Let us denote by $\Delta_i$, $i=1,2,3,4$, the determinant of the minor resulting from deleting the $i$-th column of A, and consider $f= \Delta_1 \Delta_2 \Delta_3 \Delta_4$. The polynomial $f$ defines a non-isolated hypersurface in $\C^{12}$. Therefore, from \cite{Saito94} (see also \cite{Varcenko81}), the set of all possible integral roots of $b_f(-s)$ is $\{11,10,9,8,7,6,5,4,3,2,1\}$. It is known that $\ann_{D_n[s]}(f^s) = \ann^{(1)}_{D_n[s]}(f^s)$ (see Section \ref{logannA}) and this fact can be used to simplify the computation of the annihilator.

\begin{small}
\begin{footnotesize}\begin{verbatim}
LIB "dmod.lib";
ring R = 0,(x1,x2,x3,x4,x5,x6,x7,x8,x9,x10,x11,x12),dp;
matrix A[3][4] = x1,x2,x3,x4, x5,x6,x7,x8, x9,x10,x11,x12;
poly Delta1 = det(submat(A,1..3,intvec(2,3,4)));
...                                    // analogous for Delta2 ... Delta4
poly f = Delta1*Delta2*Delta3*Delta4;
def D = Sannfslog(f); setring D;       // logarithmic annihilator
poly f = imap(R,f); number alpha = 11;
checkRoot1(LD1,f,alpha);
==> 0
\end{verbatim}\end{footnotesize}
\end{small}

Using the algorithm {\tt checkRoot} we have proved that the minimal integral root of $b_f(s)$ is $-1$. This example was suggested by F.~Castro-Jim\'enez and J.~M.~Ucha for testing the Logarithmic Comparison Theorem.
A nice introduction to this topic can be found, for instance, in~\cite{Torrelli07}.

Let $g$ be the polynomial resulting from $f$ by substituting $x_1, x_2, x_3, x_4, x_5, x_9$ with $1$. One can show that $b_g(s)$ divides $b_f(s)$ (see \cite{LM10} for details). Using the {\tt checkRoot}
algorithm we have found that
$
  (s+1)^4 (s+1/2) (s+3/2) (s+3/4) (s+5/4)
$
is a factor of $b_g(s)$ and therefore a factor of $b_f(s)$.
\end{example}

\begin{remark}
Using the notation from Section \ref{bfctIdeal}, given a holonomic $D$-module $D/I$, it is verified that $(\ini_{(-w,w)}(I) + \langle q(s) \rangle )\cap \K[s] = \langle b_{I,w}(s),q(s)\rangle$, although Theorem \ref{checkRoot} cannot be applied, since $s=\sum w_i x_i \d_i$ does not commute with all operators. For some applications like integration and restriction the maximal and the minimal integral root of the $b$-function of $I$ with respect to some weight vector have to be computed, see \cite{SST00}. However, the above formula cannot be used to find the set of all integral roots, since no upper/lower bound exists in advance. For instance, as it was suggested by N.~Takayama, $I=\langle t\d_t + k\rangle$, $k\in\Z$ is $D_1$-holonomic and one has $\ini_{(-1,1)}(I)\cap \C[s] = \langle s+k\rangle$ with $s=t\d_t$. 
\end{remark}

We close this section by mentioning that there exist some well-known methods to obtain an upper bound
for the Bernstein-Sato polynomial of a hypersurface singularity once we know, for instance, an embedded resolution of such singularity \cite{Kashiwara76/77}. Therefore using this result by Kashiwara and the {\tt checkRoot} algorithm, it is possible to compute the whole Bernstein-Sato polynomial without elimination orderings, see Example 1 in \cite{LM08}. We investigate different methods in conjunction with the further development of the \texttt{checkRoot} family of algorithms in \cite{LM10}.

\section{Bernstein operator of $f$}

\label{Boperator}

We define the Bernstein-Sato polynomial $b_f(s)$ to be the monic generator of a principal ideal, hence it is unique. But the so-called $B$-operator $P(s) \in D_n[s]$ from Theorem \ref{bfct and ann} is not unique. 

\begin{proposition}
Let $G$ be a left Gr\"obner basis of $\Ann_{D_n[s]}(f^{s+1})$ and define a \textit{Bernstein operator} to be the result of the reduced normal form $\NF(P(s), G)$ of some $B$-operator $P(s)$. Then, for a fixed monomial ordering on $D_n[s]$, the Bernstein operator is uniquely determined. 
\end{proposition}

\begin{proof}
Suppose that there is another $Q(s) \in D_n[s]$, such that the identities
\[
P(s)f^{s+1} = b_f(s)f^s  \text{ and } Q(s)f^{s+1} = b_f(s)f^s \text{ hold in the module } D_n[s]/\Ann_{D_n[s]}(f^s).
\]
Then $(P(s)-Q(s))f^{s+1} = 0$, that is $P(s)-Q(s) \in \Ann_{D_n[s]} (f^{s+1})$. Hence, the set $\{ R(s) \in D_n[s] \mid R(s)f^{s+1} = b_f(s)f^s \}$ can be viewed as an equivalence class and we can take the reduced normal form of any such operator (with respect to $G$) to be the canonical representative of the class. Since the reduced normal form with respect to a fixed monomial ordering on $D_n[s]$ is unique, so is the Bernstein operator $\NF(R(s), G)$.
\end{proof}

Note, that we can obtain the left Gr\"obner basis of $\Ann_{D_n[s]} (f^{s+1})$ via substituting $s$ with $s+1$ in the left Gr\"obner basis of $\Ann_{D_n[s]} (f^{s})$.

One can compute the Bernstein operator from the knowledge of $\Ann_{D_n[s]} (f^{s})$ and $b_f(s)$ by the following methods.

\subsection{$B$-operator via lifting}

The algorithm \textsc{Lift}$(F,G)$ computes the transformation matrix, expressing the set of polynomials $G$ via the set $F$, provided $\langle G \rangle \subseteq \langle F \rangle$. It is a classical application of Gr\"obner bases.

\begin{lemma}
Suppose, that $\Ann_{D_n[s]} (f^s)$ is generated by $h_1,\ldots,h_m$ and $b_f(s)$ is known. The output of \textsc{Lift}$(\{f,h_1,\ldots,h_m\}, \{b_f(s)\})$ is the $1\times (m+1)$ matrix $(a,b_1,\ldots,b_m)$. Then 
a $B$-operator is computed as $P(s) = \NF(a, \Ann_{D_n[s]} (f^s))$.
\end{lemma}

\begin{proof}
Because of Equation \eqref{bernstein2}, if $(a,b_1,\ldots,b_m)$ is the output of \textsc{Lift} as in the
statement,
\[
b_f(s) = af + \sum_{i=1}^m b_i h_i \text{ holds},
\]
hence the first element of such a matrix is a $B$-operator. Thus, the Bernstein operator is obtained via $\NF(a, \Ann_{D_n[s]} (f^{s+1}))$.
\end{proof}

However, we have to mention, that the \textsc{Lift} procedure is quite expensive in general.

Note that another method for the computation of a $B$-operator using lifting techniques is given by applying Algorithm~8 of \cite{NN10} with $a(x)=1$.

\subsection{$B$-operator via kernel of module homomorphism}

1. Consider the $D_n[s]$-module homomorphism
\[
\varphi: D_n[s] \lra D_n[s]/ (\Ann_{D_n[s]} (f^s) + \langle b_f(s) \rangle), \quad 1 \mapsto f,
\]
then for $u\in\ker \varphi$, $uf \in \Ann_{D_n[s]} (f^s) + \langle b_f(s) \rangle$. That is, there exist $a, b_i \in D_n[s]$, such that
\[
uf = a b_f(s) + \sum_{i=1}^m b_i h_i,
\]
However, we are interested in such $u$, that $a\in\K$. This is possible, but the 2nd method above proposes a more elegant solution. Also one has to say, that in this case we have to compute a Gr\"obner basis of $\Ann_{D_n[s]} (f^s) + \langle b_f(s) \rangle$ as an intermediate step and also the kernel of a module homomorphism with respect to the latter. This combination is, in general, quite nontrivial to compute. In the Gr\"obner basis computation a monomial ordering, preferring $x,\d x$ over $s$ seems to be better because of numerous applications of the Product Criterion.

\smallskip\noindent
2. Consider the $D_n[s]$-module homomorphism
\[
\vartheta: D_n[s]^2 \lra D_n[s]/ \Ann_{D_n[s]} (f^s), \quad \epsilon_1 \mapsto b_f(s), \epsilon_2 \mapsto f.
\] 
Then $\Ker \vartheta = \{ (u,v)^T \in D_n[s]^2 \mid u b_f(s) + v f \in \Ann_{D_n[s]} (f^s) \}$ is a submodule of $D_n[s]^2$. Indeed, $\ker \vartheta$ has many generators. In order to get a vector of the form $(k,u(s))$ for $k\in\K$, we perform another Gr\"obner basis computation for a submodule with respect to a module monomial ordering, giving preference to the first component over the second one. Since in the reduced basis there is a single element of the form $(k,v(s))\subset \Ker \vartheta$ with $k\not=0$, it follows that $P(s) = v(s) k^{-1}$.

This algorithm is implemented in \texttt{dmod.lib} as \texttt{operatorModulo}. The approach via lifting is used in the procedure \texttt{operatorBM}, which computes all the Bernstein data. The procedures can be used as follows.

\begin{example}
Consider the Reiffen curve $f = x^2 + y^3 + x y^2 \in \K[x,y]$. At first we use \texttt{operatorBM} and compare the length of an $B$-operator computed via lift to the length of the Bernstein operator.

\begin{footnotesize}\begin{verbatim}
LIB "dmod.lib";
ring r = 0,(x,y),dp;
poly F = x^2 + y^3 + x*y^2; 
def A = operatorBM(F); setring A;
size(PS);                            // size of B-operator
==> 238
ideal LD2 = subst(LD,s,s+1);         // Ann(F^{s+1})
LD2 = groebner(LD2);                 // LD is not a Groebner basis
poly PS2 = NF(PS,LD2); size(PS2);    // size of Bernstein operator
==> 41
\end{verbatim}\end{footnotesize}

\noindent So, computing with lifting potentially computes much longer operators. Let us compare with \texttt{operatorModulo}.

\begin{footnotesize}\begin{verbatim}
poly PS3 = operatorModulo(F,LD,bs); size(PS3); 
==> 41
\end{verbatim}\end{footnotesize}

\noindent The size of the operator, returned by \texttt{operatorModulo} need not be minimal (\eg by disabling some of the interactive options of \textsc{Singular} one can get a polynomial of length 50 in this example), but it is in general much shorter, than the one, delivered by \texttt{operatorBM}. Let us check the main property of the $B$-operator and print its highest terms:

\begin{footnotesize}\begin{verbatim}
NF(PS3*F - bs, subst(LD2,s,s-1)); 
==> 0
108*PS2;                            // i.e. the Bernstein operator
==> 6*x*Dx^2*Dy+9*y*Dx^2*Dy-2*x*Dx*Dy^2-4*y*Dx*Dy^2-y*Dy^3+ ...
\end{verbatim}\end{footnotesize}

\noindent In the last line we see the terms of highest degree with respect to $\d x, \d y$.
\end{example}

\subsection{Gr\"obner free method}

As a consequence of Theorem \ref{bfct and ann}, one obtains that $P \cdot f - b_f(s) \in \Ann_{D_n[s]}(f^s)$ and $P, b_f(s) \notin \Ann_{D_n[s]}(f^s)$. If we fix an ordering such that $b_f(s) = \NF(b_f(s), \Ann_{D_n[s]}(f^s))$ holds, we may rewrite this relation to $b_f(s) = \NF(P \cdot f, \Ann_{D_n[s]}(f^s))$. Hence, we can compute $P$ by searching for a linear combination of monomials $m \in D_n[s]$ that satisfy this equality when multiplied with $f$ from the right side. Using the results from the beginning of this section, one only needs to consider monomials which span $D_n/\Ann_{D_n[s]}(f^{s+1})$ as $\K$-vector space. We get the following algorithm.

\begin{algorithm}\label{BOperator} ~
\begin{algorithmic}
	\REQUIRE $f \in \K[x_1,\ldots,x_n]$, the Bernstein-Sato polynomial $b_f(s)$ of $f$
	\ENSURE $P \in D_n[s]$, the Bernstein operator of $f$
	\STATE $d := 0$
	\LOOP
		\STATE $M_d := \{ m \in D_n[s] \mid m \text{ monomial}, \deg(m) \leq d, \lm(p) \nmid m \,\forall\, p \in \Ann_{D_n[s]}(f^{s+1}) \}$
		\IF{there exist $a_m \in \K$ such that
			$b_f(s) = \sum_{m \in M_d} a_m \NF(m \cdot f,\Ann_{D_n[s]}(f^s))$}
			\RETURN $P := \sum\limits_{m \in M_d} a_m m - b_f(s)$
  	\ELSE \STATE $d := d+1$
 		\ENDIF
	\ENDLOOP
\end{algorithmic}
\end{algorithm}

The search for the coefficients $a_m$ can be done using \texttt{linReduce} (cf. Algorithm \ref{PrincipalIntersect}) as one is in fact looking for a linear dependency between the Bernstein-Sato polynomial and the elements $m \cdot f$ in the vector space $D_n[s] / \Ann_{D_n[s]}(f^s)$. 

\begin{remark}
Note, that Algorithm \ref{BOperator} can be extended to one searching for both $B$-operator and \BS polynomial simultaneously. We have to mention, that both algorithms of this kind are well suited for the search of operators and \BS polynomials in the case, when both of them are of relatively low total degree.
\end{remark}

\subsection{Computing integrals and zeta functions}

Given a simplex $C\subset\K^n$ (for $\K=\R,\C$) and $f\in \K[x_1,\ldots,x_n]$, we can define $\zeta(s) := \int_C f(x)^s dx$. Since $P(s)\bullet f^{s+1} = b_f(s)f^s$, we obtain
\[
\zeta(s) = \int_C f(x)^s dx = \frac{1}{b_f(s)} \int_C P(s)\bullet f(x)^{s+1} dx
\]
expanding the latter with \eg the chain rule, we come to an in general inhomogeneous recurrence relation for $\zeta(s)$, which involves coefficients in $\K[s]$. Since $P(s)$ is globally defined (and is, of course, independent on $C$), one can obtain a generic formula for all integrals of this type.

\begin{example}
Let $f = x^2-x \in \K[x]$. Then the Bernstein operator reads as $P(s) = (2x-1)\d_x - 4(s+1)$
and $b_f(s)=s+1$. Any simplex in $\K^1$ is the interval $[a,b]=:C$.
\[
\zeta(s) = \int_C f(x)^s dx = \frac{1}{s+1} \int_C ((2x-1)\d_x - 4(s+1))\bullet f(x)^{s+1} dx
\]
\[
= \frac{1}{s+1} \int_C (2x-1)\d_x \bullet f(x)^{s+1} dx - 4 \zeta(s+1)
\]
By the chain rule, $\int_C (2x-1) (\d_x \bullet f(x)^{s+1}) dx = (2x-1)f(x)^{s+1}\mid_C - 2\int_C f(x)^{s+1}) dx$, hence
\[
\zeta(s) = \frac{1}{s+1} \cdot (2x-1)f(x)^{s+1}\mid_C - \frac{2}{s+1}\zeta(s+1)  - 4 \zeta(s+1),
\]
and thus
\[
(4s+6) \zeta(s+1) + (s+1) \zeta(s) = (2b-1)(b^2-b)^{s+1} - (2a-1)(a^2-a)^{s+1}
\]
The right hand side, say $R(s)$, satisfies the homogeneous recurrence $R(s+2) - (a^2-a+b^2-b) R(s+1) + (a^2-a)(b^2-b) R(s) = 0$ of order 2. Substituting the left hand side into it, we obtain a
homogeneous recurrence with polynomial coefficients of order 3:
\begin{gather*}
(a^2-a)(b^2-b)(s+1) \zeta(s)
- ( (s+2)(a^2-a+b^2-b) - (4s+6)(a^2-a)(b^2-b)) \zeta(s+1)\\
- ( (4s+10)(a^2-a+b^2-b) - (s+3) ) \zeta(s+2)
+ (4s+14) \zeta(s+3) 
= 0.
\end{gather*}

To guarantee the uniqueness of a solution to this equation, we need to specify 3 initial values, which can be easily done. However, such recurrences very seldom admit a closed form solution, thus most information about $\zeta(s)$ is contained in the recurrence itself.
\end{example}

\section{Logarithmic annihilator of $f$} \label{logannA}

Given a polynomial $f\in\K[x]=\K[x_1,\ldots,x_n]$, consider the left ideal $\ann^{(1)}_{D_n[s]}(f^s) \subseteq D_n[s]$ generated by those differential operators $P(s)\in D_n[s]$ of total order (in the partials) less than or equal to one, which annihilate $f^s$. This ideal is clearly contained in $\ann_{D_n[s]}(f^s)$ and can be generated by elements of the form
$$
  P(s) = a_0(x,s) + a_1(x,s) \d_{x_1} + \cdots + a_n(x,s) \d_{x_n} \in D_n[s],
$$
where $(a_0,a_1,\ldots,a_n)\in \syz_{\K[x,s]}(f,s\frac{\d f}{\d x_1}, \cdots,s\frac{\d f}{\d x_n})$. Therefore, for each $f\in \K[x]$ one can compute, by using Gr\"obner bases in $\K[x,s]$, a system of generators of $\ann^{(1)}_{D_n[s]}(f^s)$. The corresponding procedure in \texttt{dmod.lib} is called \texttt{Sannfslog}. Let us see the Reiffen curve $f=x^4+y^5+xy^4$ as an example with {\sc Singular}.

\begin{footnotesize}\begin{verbatim}
LIB "dmod.lib";
ring R = 0,(x,y),dp;
poly f = x^4+y^5+x*y^4;
def A = Sannfslog(f); setring A; LD1;
==> LD1[1]=4*x^2*Dx+5*x*Dx*y+3*x*y*Dy-16*x*s+4*y^2*Dy-20*y*s
==> LD1[2]=16*x*Dx*y^2-125*x*Dx*y-4*x^2*Dy+4*Dx*y^3+5*x*y*Dy+12*y^3*Dy-100*y^2*Dy
 -64*y^2*s+500*y*s
// now we compute the whole annihilator with Sannfs and compare
setring R; def B = Sannfs(f); setring B;
map F = A,x,Dx,y,Dy,s;
ideal LD1 = F(LD1);
LD1 = groebner(LD1);
simplify( NF(LD,LD1), 2); 
==> _[1]=36*y^3*Dx^2-36*y^3*Dx*Dy+1125/4*x*y*Dx^2-315/4*x*y*Dx*Dy+ ...
\end{verbatim}\end{footnotesize}

\noindent And the latter polynomial is not an element of $\ann^{(1)}_{D_n[s]}(f^s)$ but of $\ann^{(2)}_{D_n[s]}(f^s) = \ann_{D_n[s]}(f^s)$.

\subsection{The annihilator up to degree $k$}

More generally, for a given $k\geq 1$ one can consider the left ideal $\ann^{(k)}_{D_n[s]}(f^s) \subseteq D_n[s]$ generated by the differential operators $P(s)\in D_n[s]$ of total order less than or equal to $k$, such that $P(s)$ annihilate $f^s$. The tower of ideals
$$
  \ann^{(1)}_{D_n[s]} (f^s) \subsetneq \cdots
  \subsetneq \ann^{(k_0)}_{D_n[s]} (f^s) = \ann_{D_n[s]} (f^s)
$$
has been recently studied by Narv\'aez in \cite{Narvaez08}. It is an open problem to find the minimal integer $k_0$ satisfying the above condition without computing the whole annihilator.

Computationally the annihilator up to degree $k$ can be obtained using Gr\"obner bases in $\K[x,s]$ as follows. Consider $P(s) = \sum_{|\beta|\leq k} a_\beta \d^{\beta} \in\ann^{(k)}_{D_n[s]}(f^s)$ and let $g_\beta(x,s) \in \K[x,s]$ be the polynomial defined by the formula $\d^\beta\cdot f^s = g_\beta \cdot f^{s-|\beta|}$. Then $(a_\beta)_{|\beta|\leq k} \in \syz (g_\beta f^{k-|\beta|})_{|\beta|\leq k}$. Eventually, the polynomials $g_\beta(x,s)$ can be computed using the expression given in Lemma \ref{formulaFS}. 

Given $\beta=(\beta_1,\ldots,\beta_n)\in\N^n$, as usual $|\beta| = \beta_1+\cdots+\beta_n$, $\beta! = \beta_1! \cdots \beta_n!$ and $\Delta_x^\beta = \frac{1}{\beta!} \d_x^\beta$. A {\em partition} of $\beta$ is a way of writing $\beta$ as a sum of integral vectors with non-negative entries. Two sums which only differ in the order of their summands are considered to be the same partition. If $\beta = \sigma_1+\ldots+\sigma_k$ with $\sigma_i\neq 0$, then $\sigma$ is said to be a partition of {\em length} $k$. The set of all partitions of $\beta$ (resp. of length $k$) is denoted by $\mathcal{P}(\beta)$ (resp. $\mathcal{P}(\beta; k)$). Obviously $\mathcal{P}(\beta) = \cup_{k=1}^{|\beta|} \mathcal{P}(\beta; k)$. Finally, we write $\ell(\sigma) := (\ell_{\sigma\tau})_{\tau}$, where $\ell_{\sigma\tau}$ is the number of times that $\tau$ appears in~$\sigma$.

\begin{lemma}\label{formulaFS}
Using the above notation for all non-zero $\beta\in\N^n$ we have,
\begin{equation*}
\begin{split}
  \Delta_x^{\beta} \cdot f^s = & 
\sum_{k=1}^{|\beta|} \binom{s}{k} \sum_{\sigma\in\mathcal{P}(\beta; k)}
  \frac{1}{\ell(\sigma)!}\ \Delta_x^{\sigma_1}(f)\cdots \Delta_x^{\sigma_k}(f)
  \cdot f^{s-k}\,.
\end{split}
\end{equation*}
\end{lemma}

This formula was suggested by Narv\'aez and can be proved by induction on $|\beta|$. Similar expressions appear in \cite[Prop. 5.3.5]{MN98} and \cite[Prop. 2.3.2]{Mebkhout97}.

However, despite this almost closed form, the set of polynomials, between which we have to compute syzygies, is growing fast and the size of polynomials increases. This results in quite hard computations even with the mentioned enhancements.

\section{Bernstein-Sato ideals for $f = f_1 \cdot \ldots \cdot f_m$} \label{BMI}

Using the results from \cite{GHU05}, which we confirmed through intensive testing (cf. \cite{LM08}), it follows, that the method by \BM is the most effective one for the computation of $s$-parametric annihilators where $f = f_1$. Because of the structure of annihilators in the situation $f = f_1 \cdots f_p$, $p>1$, basically the same principles stand behind the corresponding algorithms. Hence, we decided to implement only \BM's method for the $s$-parametric annihilator $\Ann_{D_n[s]}(f^s) \subset D_n[s]$, where $s = (s_1,\ldots,s_p)$. The corresponding procedure in \texttt{dmod.lib} is called \texttt{annfsBMI}. It computes both $\Ann_{D_n[s]} f^s \subset D_n[s]$ and the Bernstein-Sato ideal in $\K[s]$, which is defined as 
\[
\mathcal{B}(f) = 
(\Ann_{D_n[s_1,\ldots,s_p]}(f_1^{s_1} \cdots f_p^{s_p}) + \langle f_1 \cdots f_p \rangle ) \cap \K[s_1,\ldots,s_p].
\]

In contrary to the case $f=f_1$, in general the ideal $\mathcal{B}(f)$ need not be principal. However, it is an open question to give a criterion for the principality of $\mathcal{B}(f)$. Armed with such a criterion, one can apply a generalization of the method of Principal Intersection \ref{PrincipalIntersect} to multivariate subalgebras \cite{ALM10} and thus replace expensive elimination above by the computation of a minimal polynomial. Otherwise we still can apply the Principal Intersection, which, however, will deliver only one polynomial to us. As in the case $f=f_1$ it is an open question, which strategy and which orderings should one use in the computation of the annihilator and of the \BS ideal in order to achieve better performance.

We reported in \cite{LM08} on several challenges, which have been solved with the help of our implementation. Namely, the products $(x^3+y^2)(x^2+y^3)$ and $(x^2+y^2+y^3)(x^3+y^2)$ give rise to principal \BS ideals.

\begin{example}
Let us consider the following example from \cite{RT10}, which is quite challenging to compute indeed.

\begin{footnotesize}\begin{verbatim}
LIB "dmod.lib"; ring r = 0,(x,y,z),dp;
ideal F = z, x^5 + y^5 + x^2*y^3*z;
def A = annfsBMI(F); setring A;
LD; // prints the annihilator in D[s1,s2] 
BS; // prints the Bernstein-Sato ideal
\end{verbatim}\end{footnotesize}

\noindent We do not show the output here because of its size. But from the output one can see, the Bernstein-Sato ideal is of dimension 1 in $D_3[s_1,s_2]$ since its Gr\"obner basis consists of three elements $\{15625 s_1 s_2^8 + 17 \text{ l.o.t.}, 3125 s_1^2 s_2^7 + 24 \text{ l.o.t.}, 625 s_1^5 s_2^6 + 42 \text{ l.o.t.}\}$. Notably, every generator factorizes into linear factors and each factor involves either $s_1$ or $s_2$, which happens quite seldom in general.
\end{example}

In general, quite a little is known about \BS ideals. Their dimensions, principality, factorization of generators and primary decomposition constitute open problems from the theoretical side.

\section{Bernstein-Sato polynomial for a variety} \label{bsvarA}

Now we proceed to the construction of the Bernstein-Sato polynomial of an affine algebraic variety. We refer to \cite{BMS06} for the details of the complete construction for arbitrary varieties and to \cite{ALM09} for the details about annihilator-driven algorithms.

Given two positive integers $n$ and $r$, for the rest of this section we fix the indices $i,j,k,l$ ranging between $1$ and $r$ and an index $m$ ranging between $1$ and $n$.Let $f=(f_1,\ldots,f_r)$ be an $r$-tuple in $\K[x]^r$. Here, $s=(s_1,\ldots,s_r)$, $\frac{1}{f} = \frac{1}{f_1\cdots f_r}$ and $f^s = f_1^{s_1}\cdots f_r^{s_r}$. Let us denote by $\K\langle S \rangle$ the universal enveloping algebra $U(\mathfrak{gl}_{\,r})$, generated by the set of variables $S=(s_{ij})$, $i,j=1,\ldots,r$, with $s_{ii} = s_i$, subject to relations $[s_{ij}, s_{kl}] = \delta_{jk} s_{il} - \delta_{il} s_{kj}$. We denote by $D_n\langle S \rangle := D_n \otimes_{\K} \K\langle S \rangle$, which is a $G$-algebra of Lie type by \eg \cite{LS03}. Then the free $\K[x,s,\frac{1}{f}]$-module of rank one generated by the formal symbol $f^s$ has a natural structure of a left $D_n \langle S \rangle$-module:
\[
 s_{ij} \bullet (G(s)\cdot f^s) = s_i \cdot G(s+\epsilon_j-\epsilon_i)
 \frac{f_j}{f_i} \cdot f^s \in \K[x,s,\frac{1}{f}]\cdot f^s,
\]
where $G(s) \in \K[x,s,\frac{1}{f}]$ and $\epsilon_j$ stands for the $j$-th standard basis vector.

\begin{theorem}[Budur et al. \cite{BMS06}]
For every $r$-tuple $f=(f_1,\ldots,f_r)\in \K[x]^r$ there exists a non-zero polynomial in one variable $b(\vars)\in \K[\vars]$ and $r$ differential operators $P_1(S),\ldots,P_r(S)\in D_n\langle S \rangle$ such that
\begin{equation}\label{BSvariety}
\sum_{k=1}^r P_k(S) f_k \cdot f^s = b(s_1+\cdots+s_r) \cdot f^s \in \K[x,s,\frac{1}{f}]\cdot f^s.
\end{equation}
\end{theorem}

The {\em Bernstein-Sato polynomial} $b_f(\vars)$ of $f=(f_1,\ldots,f_r)$ is defined to be the monic polynomial of lowest degree in the variable $\vars$ satisfying the equation \eqref{BSvariety}. It turns out, that every root of the \BS polynomial is rational, as in the case of a hypersurface. Let $I$ be the ideal generated by $f_1,\ldots, f_r$ and $Z$ the affine algebraic variety associated with $I$ in $\K^n$. Then it can be verified, that $b_f(\vars)$ is independent of the choice of a system of generators of $I$, and moreover that $b_Z(\vars) = b_f(\vars - \codim Z + 1)$ depends only on $Z$. 

Now, let us denote by $\ann_{D_n\langle S\rangle}(f^s)$ the left ideal of all elements $P(S) \in D_n\langle S \rangle$ such that $P(S)\bullet f^s = 0$. We call this ideal the {\em annihilator} of $f^s$ in $D_n \langle S \rangle$. From the definition of the \BS polynomial it becomes clear that
$$
  (\ann_{D_n \langle S \rangle} (f^s) + \langle f_1,\ldots,f_r \rangle)\cap
  \K[s_1+\cdots+s_r] = \langle b_f(s_1+\ldots+s_r)\rangle.
$$
Since the final intersection can be computed with Principal Intersection \ref{PrincipalIntersect}, the above formula provides an algorithm for computing the Bernstein-Sato polynomial of affine algebraic varieties, once we know a Gr\"obner basis of the annihilator of $f^s$ in $D_n \langle S \rangle$. 

\begin{theorem}\label{SannFsVar}
Let $f=(f_1,\ldots,f_r)$ be an $r$-tuple in $\K[x]^r$ and $D_n \langle \partial t, S \rangle$ the $\K$-algebra generated by $D_n$, $\partial t$ and $S$ with the non-commutative relations of $D_n\langle S \rangle$, described above and additional relations $[s_{ij},\d {t_k}]= \delta_{jk} \d {t_i}$ ($\d {t_i}$ commute mutually with the subalgebra $D_n$). Then the annihilator of $f^s$ in $D_n \langle S\rangle$ can be expressed as follows:
$$
  \bigg[ D_n \langle \partial t, S \rangle
  \Big( s_{ij} + \d {t_i} f_j \, 
,\ \partial_{m} + \sum_{k=1}^r
  \frac{\partial f_k}{\partial x_m} \partial {t_k} \left|
  \begin{array}{c} 1\leq i,j\leq r \\ 1\leq m\leq n \end{array}\right.
  \Big)\bigg] \cap D_n \langle S\rangle.
$$
\end{theorem}

Note, that this result and its proof \cite{ALM09} can be presented as natural generalization of the algorithm for computing $\ann_{D_n[s]}(f^s)$ with the method of \BM (cf. Section \ref{sAnn}). 

As Budur et~al. point out \cite[p. 794]{BMS06}, the Bernstein-Sato polynomial for varieties coincides, up to shift of variables, with the $b$-function in \cite[p.\ 194]{SST00}, if the weight vector is chosen appropriately, see also \cite{Shibuta08}. Algorithms for computing the $b$-function have been already discussed in Section \ref{bfctIdeal}, so the procedure \texttt{bfctIdeal} can be immediately applied to this situation. Hence, like for the case of a hypersurface, we have two essentially different ways to compute \BS polynomials for varieties. The comparison of these two methods is the subject of further research.

In the new \textsc{Singular} library \texttt{dmodvar.lib}\footnote{it will be distributed with the next release of {\sc Singular}}, we present the implementations of the following algorithms
\begin{itemize}
\item[] \texttt{SannfsVar}, which computes $\ann_{D_n \langle S \rangle} (f^s)$ according to the Theorem \ref{SannFsVar},
\item[] \texttt{bfctVarIn}, which computes $b_f(s_1+\ldots+s_r)$ using initial ideal approach,
\item[] \texttt{bfctVarAnn}, which computes $b_f(s_1+\ldots+s_r)$ using annihilator-driven approach.
\end{itemize}

\begin{example}
Let $TX = V(x_0^2+y_0^3, 2x_0 x_1 + 3 y_0^2 y_1) \subset \C^4$ the tangent bundle of $X = V(x^2+y^3)\subset \C^2$. Then the Bernstein-Sato polynomial of $TX$ can be computed as follows:

\begin{footnotesize}\begin{verbatim}
LIB "dmodvar.lib"; 
ring R = 0,(x0,x1,y0,y1),Dp;
ideal F = x0^2+y0^3, 2*x0*x1+3*y0^2*y1;
bfctVarAnn(F); // annihilator-driven approach
// alternatiely, one can run
bfctVarIn(F); // approach via initial ideal
\end{verbatim}
\end{footnotesize}

\noindent In both cases we obtain the polynomial
$$
  b_{TX}(\vars) = (\vars+1)^2 (\vars+1/3)^2 (\vars+2/3)^2 (\vars+1/2) (\vars+5/6) (\vars+7/6).
$$
\noindent The annihilator ideal can be computed via executing, in addition to the first 3 lines of
the above code, the following code:

\begin{footnotesize}\begin{verbatim}
def S = SannfsVar(F); // returns a ring
setring S; // in this ring, ideal LD is the annihilator
option(redSB); LD = groebner(LD); // reduced GB of LD
\end{verbatim}
\end{footnotesize}

\noindent There are 15 generators in the Gr\"obner basis of $\ann F^s$:

\medskip
\noindent $3y_{0}^{2}\d x_{1}-2x_{0}\d y_{1},\quad
3y_{0}^{2}\d x_{0}+6y_{0}y_{1}\d x_{1}-2x_{0}\d y_{0}-2x_{1}\d y_{1},\quad 
x_{0}y_{0}\d x_{1}\d y_{0}-x_{0}y_{0}\d x_{0}\d y_{1}+x_{1}y_{0}\d x_{1}\d y_{1}-2x_{0}y_{1}\d x_{1}\d y_{1},\quad
3y_{0}y_{1}\d x_{1}^{2}-x_{0}\d x_{1}\d y_{0}+x_{0}\d x_{0}\d y_{1}-x_{1}\d x_{1}\d y_{1},\quad 
3x_{0}y_{0}y_{1}\d x_{0}\d x_{1}\d y_{1}+6x_{0}y_{1}^{2}\d x_{1}^{2}\d y_{1}-x_{0}^{2}\d x_{1}\d y_{0}^{2}+x_{0}^{2}\d x_{0}\d y_{0}\d y_{1}-2x_{0}x_{1}\d x_{1}\d y_{0}\d y_{1}+x_{0}x_{1}\d x_{0}\d y_{1}^{2}-x_{1}^{2}\d x_{1}\d y_{1}^{2}+3x_{1}y_{0}\d x_{1}^{2}+3x_{0}y_{1}\d x_{1}^{2}-3y_{0}y_{1}\d x_{1}\d y_{1},\quad 
6x_{0}y_{1}^{2}\d x_{1}^{2}\d y_{0}\d y_{1}+3x_{0}y_{0}y_{1}\d x_{0}^{2}\d y_{1}^{2}-3x_{1}y_{0}y_{1}\d x_{0}\d x_{1}\d y_{1}^{2}+6x_{0}y_{1}^{2}\d x_{0}\d x_{1}\d y_{1}^{2}-x_{0}^{2}\d x_{1}\d y_{0}^{3}+x_{0}^{2}\d x_{0}\d y_{0}^{2}\d y_{1}-2x_{0}x_{1}\d x_{1}\d y_{0}^{2}\d y_{1}+x_{0}x_{1}\d x_{0}\d y_{0}\d y_{1}^{2}-x_{1}^{2}\d x_{1}\d y_{0}\d y_{1}^{2}+3x_{1}y_{0}\d x_{1}^{2}\d y_{0}+3x_{0}y_{1}\d x_{1}^{2}\d y_{0}+9x_{0}y_{1}\d x_{0}\d x_{1}\d y_{1}-6y_{0}y_{1}\d x_{1}\d y_{0}\d y_{1}+3y_{0}y_{1}\d x_{0}\d y_{1}^{2}+6y_{1}^{2}\d x_{1}\d y_{1}^{2}+3x_{1}\d x_{1}^{2}+3y_{1}\d x_{1}\d y_{1},\quad 
6x_{0}y_{1}^{2}\d x_{1}^{3}\d y_{1}-x_{0}^{2}\d x_{1}^{2}\d y_{0}^{2}+2x_{0}^{2}\d x_{0}\d x_{1}\d y_{0}\d y_{1}-2x_{0}x_{1}\d x_{1}^{2}\d y_{0}\d y_{1}-x_{0}^{2}\d x_{0}^{2}\d y_{1}^{2}+2x_{0}x_{1}\d x_{0}\d x_{1}\d y_{1}^{2}-x_{1}^{2}\d x_{1}^{2}\d y_{1}^{2}-3x_{0}y_{0}\d x_{0}\d x_{1}^{2}+3x_{1}y_{0}\d x_{1}^{3}+3x_{0}y_{1}\d x_{1}^{3}-2x_{0}\d x_{1}\d y_{0}\d y_{1}+x_{0}\d x_{0}\d y_{1}^{2}-3x_{1}\d x_{1}\d y_{1}^{2}+6y_{0}\d x_{1}^{2},\quad
s_{22}-x_{1}\d x_{1}-y_{1}\d y_{1},\quad 
6s_{21}-3x_{0}\d x_{1}-2y_{0}\d y_{1},\quad 
6s_{11}-3x_{0}\d x_{0}+3x_{1}\d x_{1}-2y_{0}\d y_{0}+4y_{1}\d y_{1},\quad 
s_{12}x_{0}+3y_{0}y_{1}^{2}\d x_{1}-x_{0}x_{1}\d x_{0}+x_{1}^{2}\d x_{1}-x_{0}y_{1}\d y_{0},\quad 
s_{12}y_{0}\d y_{1}-x_{1}y_{0}\d x_{1}\d y_{0}+2x_{1}y_{1}\d x_{1}\d y_{1}-y_{0}y_{1}\d y_{0}\d y_{1}+2y_{1}^{2}\d y_{1}^{2}-y_{0}\d y_{0}+4y_{1}\d y_{1},\quad 
3s_{12} y_0^2+6 x_1 y_0 y_1 \d x_1-3 y_0^2 y_1 \d y_0+6 y_0 y_1^2 \d y_1-2 x_0 x_1 \d y_0,\quad
s_{12}x_{1}\d y_{1}^{2}-3s_{12}y_{0}\d x_{1}+3x_{1}y_{0}y_{1}\d x_{0}\d x_{1}\d y_{1}+6x_{1}y_{1}^{2}\d x_{1}^{2}\d y_{1}-3y_{0}y_{1}^{2}\d x_{1}\d y_{0}\d y_{1}+3y_{0}y_{1}^{2}\d x_{0}\d y_{1}^{2}+6y_{1}^{3}\d x_{1}\d y_{1}^{2}-x_{0}x_{1}\d x_{1}\d y_{0}^{2}+x_{0}x_{1}\d x_{0}\d y_{0}\d y_{1}-2x_{1}^{2}\d x_{1}\d y_{0}\d y_{1}-x_{1}y_{1}\d y_{0}\d y_{1}^{2}+3x_{1}y_{0}\d x_{0}\d x_{1}+3x_{1}y_{1}\d x_{1}^{2}-6y_{0}y_{1}\d x_{1}\d y_{0}+9y_{0}y_{1}\d x_{0}\d y_{1}+18y_{1}^{2}\d x_{1}\d y_{1}-2x_{1}\d y_{0}\d y_{1}+3y_{0}\d x_{0},\quad 
3s_{12}y_{0}y_{1}\d x_{1}-s_{12}x_{1}\d y_{1}-3x_{1}y_{0}y_{1}\d x_{0}\d x_{1}-3y_{0}y_{1}^{2}\d x_{0}\d y_{1}+x_{1}^{2}\d x_{1}\d y_{0}+x_{1}y_{1}\d y_{0}\d y_{1}-3y_{0}y_{1}\d x_{0}+x_{1}\d y_{0}$.

\medskip
\noindent This ideal belongs to the $\K$-algebra $D_4\langle S \rangle$ in 12 variables, as defined in the
beginning of this section. By executing

\begin{footnotesize}\begin{verbatim}
GKdim(LD);
\end{verbatim}
\end{footnotesize}

\noindent we obtain, that the Gel'fand-Kirillov dimension of $D_4\langle S \rangle / \ann F^s$ is 6, the
half of the Gel'fand-Kirillov dimension of $D_4\langle S \rangle$. However, $D_4\langle S \rangle / \ann F^s$ is not a generalized holonomic $D_4\langle S \rangle$-module, since the annihilator of this module contains a central element $s_{12} s_{21}-s_{11} s_{22}-s_{11}$ and hence is not zero.
\end{example}

\section*{Acknowledgements}

We would like to thank Francisco Castro-Jim\'{e}nez, Jos\'{e}-Mar\'{i}a Ucha, Gert-Martin Greuel,
Enrique Artal, Jos\'{e}-Ignacio Cogolludo and Luis Narv\'aez for fruitful discussions and insightful remarks concerning our work.

\vspace{0.5cm}
\noindent Daniel Andres\\
Lehrstuhl D für Mathematik, RWTH Aachen, Templergraben 64, 52062 Aachen, Germany\\
e-mail: \texttt{Daniel.Andres@math.rwth-aachen.de}

\vspace{0.5cm}
\noindent Michael Brickenstein\\
Mathematisches Forschungsinstitut Oberwolfach, Schwarzwaldstr. 9-11, 77709 Oberwolfach-Walke, Germany\\
e-mail: \texttt{brickenstein@mfo.de}

\vspace{0.5cm}
\noindent Viktor Levandovskyy\\
Lehrstuhl D für Mathematik, RWTH Aachen, Templergraben 64, 52062 Aachen, Germany\\
e-mail: \texttt{Viktor.Levandovskyy@math.rwth-aachen.de}

\vspace{0.5cm}
\noindent Jorge Martín-Morales\\
Department of Mathematics-I.U.M.A., University of Zaragoza, C/ Pedro Cerbuna, 12 - 50009, Zaragoza, Spain\\
e-mail: \texttt{jorge@unizar.es}

\vspace{0.5cm}
\noindent Hans Schönemann\\
Fachbereich Mathematik, TU Kaiserslautern, Erwin-Schrödinger-Str. 48, 67632 Kaiserslautern, Germany\\
e-mail: \texttt{hannes@mathematik.uni-kl.de}

\end{document}